\documentclass{amsart}
\usepackage{amsfonts,amssymb,amsmath,amsthm}
\usepackage{url}
\usepackage{enumerate}
\usepackage{epsfig}
\usepackage{tikz}
\usepackage{ulem}
\usetikzlibrary{shapes,arrows}

\newtheorem{theorem}{Theorem}[section]
\newtheorem{lemma}[theorem]{Lemma}
\newtheorem{corollary}[theorem]{Corollary}
\newtheorem{proposition}[theorem]{Proposition}

\theoremstyle{definition}
\newtheorem{definition}[theorem]{Definition}
\newtheorem{example}[theorem]{Example}

\theoremstyle{remark}
\newtheorem{remark}[theorem]{Remark}

\numberwithin{equation}{section}

\tikzstyle{block} = [rectangle, draw, text width=5.3em, text centered, rounded corners, minimum height=4em]
\tikzstyle{blocknob} = [text width=4.4em, text centered, minimum height=4em]
\tikzstyle{line} = [draw, -latex']

\def\logmea {\mathrm{logmea}\,}

\def\W {\mathrm{W}}
\def\D {\mathcal{D}}
\def\A {\mathcal{A}}

\begin{document}

\title[Wiman-Valiron theory for Wilson series]{Wiman-Valiron theory for a polynomial series based on the Wilson operator}


\author[K. H. CHENG]{Kam Hang CHENG}
\address{Department of Mathematics, The Hong Kong University of Science and Technology, Clear Water Bay, Kowloon, Hong Kong.}
\email{henry.cheng@family.ust.hk}
\thanks{The author was partially supported by GRF no. 16306315 from the Research Grants Council of Hong Kong.}



\date{Dec 13, 2018.\ \ To appear in Journal of Approximation Theory.}



\begin{abstract}
We establish a Wiman-Valiron theory for a polynomial series based on the Wilson operator $\D_\W$.\ \ For an entire function $f$ of order smaller than $\frac13$, this theory includes (i) an estimate which shows that $f$ behaves locally like a polynomial consisting of the terms near the maximal term in its Wilson series expansion, and (ii) an estimate of $\D_\W^n f$ compared to $f$.\ \ We then apply this theory in studying the growth of entire solutions to difference equations involving the Wilson operator.
\end{abstract}

\maketitle
\setcounter{tocdepth}{1}
\tableofcontents

\section{Introduction}

The \textbf{\textit{order}} of an entire function $f:\mathbb{C}\to\mathbb{C}$ is defined by
\[
	\sigma\equiv\sigma_f:=\limsup_{r\to\infty}\frac{\ln^+\ln^+ M(r;f)}{\ln r},
\]
where $\displaystyle M(r;f):=\max\{|f(x)|:|x|=r\}$, and the \textbf{\textit{type}} of an entire function of order $\sigma\in(0,+\infty)$ is defined by $\displaystyle\limsup_{r\to\infty}{\frac{\ln^+ M(r;f)}{r^\sigma}}$ \cite{Boas}.\ \ In this paper, we will first show that every entire function $f$ of small order admits a \textit{Wilson series expansion}
\[
	f(x)=\sum_{k=0}^\infty{a_k\tau_k(x;0)}.
\]
The main purpose of this paper is then to show that every entire function $f$ of small order is mostly contributed by just a few terms in this Wilson series expansion which are near the \textit{maximal term} $a_N\tau_N(x;0)$.\ \ With $\D_\W$ denoting the \textit{Wilson divided-difference operator} and $\D_\W^n$ denoting that we apply $\D_\W$ for $n$ times, we will deduce from this result that
\[
	(\D_\W^n f)(x) \hspace{15px}\mbox{is asymptotically similar to}\hspace{15px} \left(\frac{N}{x}\right)^n f(x),
\]
which can be applied in analyzing the growth of entire solutions to difference equations involving the Wilson operator.\ \ This is parallel to the classical Wiman-Valiron theory for the Maclaurin series $f(z)=\sum a_k z^k$ having maximal term $a_N z^N$, from which one can deduce that
\[
	f^{(n)}(z) \hspace{15px}\mbox{is asymptotically similar to}\hspace{15px} \left(\frac{N}{z}\right)^n f(z).
\]

\bigskip
In the following introduction, we briefly review the history of the study of the Wilson operator as well as the classical Wiman-Valiron theory, and we slightly clarify the terminologies stated in the previous paragraph.\ \ Classical hypergeometric orthogonal polynomials have been an active research topic in the recent decades.\ \ The Wilson polynomials $W_n(\cdot; a,b,c,d)$, in particular, are hypergeometric orthogonal polynomials located at the top level of the Askey scheme \cite{Koekoek-Swarttouw}, and are widely recognized as the most general hypergeometric orthogonal polynomials that contain all the known classical hypergeometric orthogonal polynomials as special cases  \cite{Andrews-Askey}.\ \ With four complex parameters $a$, $b$, $c$ and $d$, these polynomials are defined, via a terminating Gauss hypergeometric series, by
\begin{align*}
	&\ \ \ \,\frac{W_n(x;\, a,\, b,\, c,\, d)}{(a+b)_n(a+c)_n(a+d)_n} \\
	&:={}_4F_3\left(
	\begin{matrix}
	\begin{matrix}
		-n, & n+a+b+c+d-1, & a+i\sqrt{x}, & a-i\sqrt{x}
	\end{matrix} \\
	\begin{matrix}
		a+b, & a+c, & a+d
	\end{matrix}
	\end{matrix}
	\ ;\ 1\right),
\end{align*}
where the notation $(a)_k:=a(a+1)\cdots(a+k-1)$ denotes the rising $k$-step factorial of a complex number $a$.\ \ The Wilson operator $\D_\W$ was first considered by Wilson \cite[p. 34]{Askey-Wilson} to study these polynomials.\ \ This operator acts on Wilson polynomials in a similar manner as the usual differential operator acts on monomials, except with a shift in the four parameters:
\begin{align*}
	(\D_\W W_n)(x;a,b,c,d)=C_n W_{n-1}\Big(x;a+\frac{1}{2},b+\frac{1}{2},c+\frac{1}{2},d+\frac{1}{2}\Big),
\end{align*}
where $C_n=-n(n+a+b+c+d-1)$.

\bigskip
Research work regarding the Wilson operator so far has mainly focused on its interactions with Wilson polynomials.\ \ Therefore the author and Chiang have recently looked into this operator in a broader function-theoretic context, and established some results about its interaction with meromorphic functions, including a Nevanlinna theory and a Picard-type theorem \cite{Cheng-Chiang}.\ \ In this paper, we will turn our focus to entire functions and their series expansions.\ \ In 1926, N\o rlund has written a memoir on interpolation series \cite{Noerlund}, a large part of which was devoted on investigating how to expand entire functions into various interpolation series, e.g. Newton series and Stirling series.\ \ As opposed to the Taylor series which converges on disks, the Newton series has its regions of convergence to be right half-planes, while that for the Stirling series is effectively the whole complex plane.\ \ Stirling series converge faster than Newton series, and the interpolating polynomials in Stirling series are obtained by slightly modifying the Wilson polynomials.\ \ Thus it is natural to develop a function theory for an interpolation series using the Wilson polynomials $\{W_n\}$ as a basis, instead of the usual basis $\{(x-x_0)^n\}$ of the Taylor series.\ \ In practice we simply use the basis $\{\tau_n(\cdot;x_0)\}$, where
\[
	\tau_n(x;a^2) := \prod_{k=0}^{n-1}{[(a+ki)^2-x]}
\]
is modified from the key factors $(a+i\sqrt{x})_n(a-i\sqrt{x})_n$ in $W_n(x;a,b,c,d)$.\ \ In fact, following the classical idea as in \cite{Noerlund}, we will show in this paper that every entire function $f$ satisfying the growth condition
\[
	\limsup_{r\to\infty}{\frac{\ln^+ M(r;f)}{\sqrt{r}}} < 2\ln 2
\]
admits, for each $x_0\in\mathbb{C}$, a \textit{Wilson series expansion}
\[
	\sum_{k=0}^\infty{a_k\tau_k(x;x_0)}
\]
which converges uniformly to $f$ on any compact subset of $\mathbb{C}$, where the coefficients $a_k$ are generated by the values of $f$ at a sequence of interpolation points.

\bigskip
In the 1910s, Wiman \cite{Wiman1, Wiman2} introduced a theory which relates the \textit{maximum modulus} $M(r;f)$ of an entire function $f$ on a circle $\partial D(0;r)$ to the \textit{maximal term} $\mu(r;f)$ of its Maclaurin series expansion on the circle.\ \ Here given an entire function $\displaystyle f(z)=\sum_{k=0}^\infty{a_kz^k}$, one defines the \textit{maximal term} and the \textit{central index} of $f$ to be respectively the functions $\mu(\cdot;f):[0,+\infty)\to[0,+\infty)$ and $\nu(\cdot;f):(0,+\infty)\to\mathbb{N}_0$, given by
\[
	\mu(r;f) :=\max_{n\in\mathbb{N}_0}|a_n|r^n
\]
and
\[
	\nu(r;f) :=\max\{n\in\mathbb{N}_0:|a_n|r^n=\mu(r;f)\},
\]
so that for every $r>0$ one simply has $\mu(r;f)=|a_N|r^N$ where $N=\nu(r;f)$.\ \ In the following decades, Wiman's theory was then developed more extensively by Valiron \cite{Valiron1, Valiron2, Valiron}, Saxer \cite{Saxer}, Clunie \cite{Clunie1, Clunie2} and K\"{o}vari \cite{Kovari1, Kovari2}.\ \ It turned out that using this theory one can describe the local behavior of an entire function using its maximal term.\ \ The main idea, in Clunie's form \cite{Clunie1}, is that given any entire function $\displaystyle f(z)=\sum_{k=0}^\infty{a_kz^k}$ of finite order, the quantity
\[
	\left|\sum_{k:|k-N|>\kappa}{a_kz^k}\right|,
\]
which is the modulus of the tail of its Maclaurin series expansion, is small relative to $\mu(r;f)$ as $r=|z|\to\infty$ outside a certain small exceptional set, where $\kappa$ is a small positive integer that depends on the central index $N$.\ \ Moreover, one can even obtain from this theory a local relationship between $f$ and its derivative $f'$, given by
\[
	\frac{z}{N}f'(z)=f(z)(1+o(1))
\]
as $|z|=r\to\infty$ outside a certain small exceptional set.\ \ With different expressions of the $o(1)$, this holds for all $z\in\partial D(0;r)$ in Clunie \cite{Clunie2}, and holds just for those $z\in\partial D(0;r)$ taking values $|f(z)|$ close to the maximum $M(r;f)$ in Valiron \cite{Valiron}, Clunie \cite{Clunie1} and Saxer \cite{Saxer}.\ \ Fenton \cite{Fenton1} has extended the theory to entire functions of finite lower order.\ \ Apart from the chapter in Valiron's book \cite{Valiron3} which summarized much of his work on this topic, Hayman has written a comprehensive survey \cite{Hayman2} on the theory, and Fenton has also written a short summary \cite{Fenton} that has made the theory easy to access.\ \ One can also see \cite{Jank-Volkmann} and \cite{He-Xiao} for more modern references of this theory.

\bigskip
Recently, the Wiman-Valiron theory has been extended by Ishizaki and Yanagihara \cite{IY} to the case of Newton series and the ordinary difference operator, and Chiang and Feng \cite{Chiang-Feng4} have also given a Wiman-Valiron estimate of successive ordinary differences.\ \ Ishizaki and Yanagihara's result \cite{IY} is that for any entire function $\displaystyle f(z)=\sum_{k=0}^\infty{a_kz(z-1)\cdots(z-k+1)}$ of order smaller than $\frac{1}{2}$, the quantity
\[
	\sum_{k:|k-N|>\kappa}{k^n|a_k|r(r+1)\cdots(r+k-1)}
\]
is small relative to $N^n\mu^*(r;f)$ as $r\to\infty$ outside a certain small exceptional set, where $\mu^*(r;f)$ and $N=\nu^*(r;f)$ are Newton series analogues of the maximal term and the central index of $f$, and $\kappa$ is a small positive integer that depends on $N$.\ \ This result enables us to deduce a local behavior of successive ordinary differences of $f$, namely that
\[
	(\Delta^n f)(z) \hspace{15px}\mbox{is asymptotically similar to}\hspace{15px} \left(\frac{N}{z}\right)^n f(z)
\]
as $r=|z|\to\infty$ outside the same small exceptional set.

\bigskip
Now in this paper we aim to develop an analogous theory for the Wilson series, i.e., to investigate how the local behavior of an entire function is controlled by the terms in its Wilson series expansion and to look into the behavior of successive Wilson differences of entire functions.\ \ In particular, the main result (Theorem~\ref{tail}) is that for any entire function $\displaystyle f(x)=\sum_{k=0}^\infty{a_k\tau_k(x;0)}$ of order smaller than $\frac{1}{3}$, the quantity
\[
	\sum_{k:|k-N|>\kappa}{k^n|a_k\tau_k(r;0)|}
\]
is small relative to $N^n\mu_\W(r;f)$ as $r\to\infty$ outside a certain small exceptional set, where $\mu_\W(r;f)$ and $N=\nu_\W(r;f)$ are Wilson series analogues of the maximal term and the central index of $f$, and $\kappa$ is again a small positive integer that depends on $N$.\ \ This gives a local behavior of successive Wilson differences of $f$ (see Theorem~\ref{WVmain}), which is that
\[
	(\D_\W^n f)(x) \hspace{15px}\mbox{is asymptotically similar to}\hspace{15px} \left(\frac{N}{x}\right)^n f(x)
\]
as $r=|x|\to\infty$ outside the same small exceptional set.\ \ Along the development of this theory, we have also obtained some new results.\ \ These include a Leibniz rule for the Wilson operator (see Theorem~\ref{WLeibniz}) and a result analogous to Lindel\"{o}f-Pringsheim Theorem (see Theorem~\ref{WLP}) which relates the order of the maximal term of a Wilson series to the rate of decay of its coefficients.\ \ These results on the Wilson calculus may potentially have independent interest in combinatorics or number theory.\ \ We remark here that part of the results obtained in this paper is contained in the PhD thesis of the author \cite{Cheng}.

\bigskip
This paper is organized as follows.\ \ In \S\ref{sec:DW}, we will first give the definition and some basic properties of the Wilson operator $\D_\W$, including its Leibniz rule.\ \ We will also prove the Wilson Series Theorem, which states that a function satisfying the aforementioned sufficient growth condition will admit a Wilson series expansion that converges uniformly on compact sets to the function itself.\ \ In the subsequent sections, we will then develop the Wiman-Valiron theory of $\D_\W$.\ \ We will state in \S\ref{sec:Mainresults} our main results which include two theorems.\ \ The first one asserts that the local behavior of an entire function is mainly contributed by those terms in its Wilson series expansion that are near the maximal term; and the second one gives an estimate of $\D_\W^n f$ compared to $f$.\ \ Before proving these two theorems in \S\ref{sec:Proofs}, we will establish some properties of the Wilson maximal term and central index in \S\ref{sec:Propmunu}.\ \ Finally, the main results will be applied in \S\ref{sec:Applications} to give estimates on the growth of transcendental entire solutions to a certain type of Wilson difference equations and on the growth of its Wilson maximal term.

\bigskip
In this paper, we adopt the following notations:
\begin{enumerate}[(i)]
	\item $\mathbb{N}$ denotes the set of all natural numbers \textit{excluding} $0$, and $\mathbb{N}_0:=\mathbb{N}\cup\{0\}$.
	\item For every positive real number $r$ and every complex number $a$, $D(a;r)$ denotes the open disk of radius $r$ centered at $a$ in the complex plane.
	\item For every positive real number $r$, we denote $\ln^+ r:=\max\{\ln r,0\}$.
	\item A \textit{complex function} always means a function in \textit{one} complex variable, and an \textit{entire function} always means a holomorphic function from $\mathbb{C}$ to $\mathbb{C}$, unless otherwise specified.
	\item A summation notation of the form $\displaystyle\sum_{k:S_k}$ denotes a sum running over all the $k$'s such that the statement $S_k$ is true.
	\item For any two functions $f,g:[0,\infty)\to\mathbb{R}$, we write
	\begin{itemize}
		\item $g(r)=O(f(r))$ as $r\to\infty$ if and only if there exist $C>0$ and $M>0$ such that $|g(r)|\le C|f(r)|$ whenever $r>M$;
		\item $g(r)=o(f(r))$ as $r\to\infty$ if and only if for every $C>0$, there exists $M>0$ such that $|g(r)|\le C|f(r)|$ whenever $r>M$;
	\end{itemize}
\end{enumerate}

\pagebreak
\section{The Wilson operator and the Wilson basis}
\label{sec:DW}

In this section, we give the definition of the Wilson operator and a few of its properties.

\begin{definition}
	Let $\sqrt{\cdot}$ be a branch of the complex square-root with the imaginary axis as the branch cut.\ \ For each $x\in\mathbb{C}$ we denote
	\[
		x^+ := \left(\sqrt{x}+\frac{i}{2}\right)^2 \hspace{30px} \mbox{and} \hspace{30px} x^- := \left(\sqrt{x}-\frac{i}{2}\right)^2.
	\]
	We also denote $x^{+(0)}=x^{-(0)}:=x$, $x^{\pm(m)}:=(x^{\pm(m-1)})^\pm$ and $x^{\pm(-m)}:=x^{\mp(m)}$ for every positive integer $m$.\ \ Then we define the \textbf{\textit{Wilson operator}} $\D_\W$, which acts on all complex functions, as follows:
	\begin{equation}
		\label{A1} (\D_\W f)(x) := \frac{f(x^+)-f(x^-)}{x^+ - x^-} = \frac{f((\sqrt{x}+\frac{i}{2})^2)-f((\sqrt{x}-\frac{i}{2})^2)}{2i\sqrt{x}}.
	\end{equation}
	We also define the \textbf{\textit{Wilson averaging operator}} $\A_\W$ by
	\begin{equation}
		\label{A2} (\A_\W f)(x) := \frac{f(x^+)+f(x^-)}{2} = \frac{f((\sqrt{x}+\frac{i}{2})^2) + f((\sqrt{x}-\frac{i}{2})^2)}{2}.
	\end{equation}
\end{definition}

According to \eqref{A1} and \eqref{A2}, although there are two choices of $\sqrt{x}$ for each $x\ne 0$, $\D_\W$ and $\A_\W$ are independent of the choice of $\sqrt{\cdot}$ and are thus always well-defined.\ \ Moreover, the value of $\D_\W f$ at $0$ should be defined as
	\[
		(\D_\W f)(0) := \lim_{x\to 0}(\D_\W f)(x) = f'\left(-\frac{1}{4}\right),
	\]
in case $f$ is differentiable at $-\frac{1}{4}$.\ \ We sometimes write $z:=\sqrt{x}$.\ \ We consider branches of square-root with the imaginary axis as the branch cut, because of the shift of $\frac{i}{2}$ in the definition of $\D_\W$.

\bigskip
We first look at some algebraic properties of the Wilson operator.\ \ First of all, it is apparent that the Wilson operator has the following product and quotient rules.
\begin{lemma}
\textup{(Wilson product and quotient rules)}
\label{PQRule}
	For every pair of complex functions $f$ and $g$, we have
	\begin{enumerate}[(i)]
		\item
			$\displaystyle (\D_\W(fg))(x) = (\A_\W f)(x)(\D_\W g)(x) + (\D_\W f)(x)(\A_\W g)(x)$, and
		\item
			$\displaystyle \left(\D_\W\frac{f}{g}\right)(x) = \frac{(\D_\W f)(x)(\A_\W g)(x) - (\A_\W f)(x) (\D_\W g)(x)}{g(x^+)g(x^-)}$ provided that $g\not\equiv 0$.
	\end{enumerate}
\end{lemma}

More generally, we have the following Leibniz rule for the Wilson operator.

\begin{theorem}
\textup{(Wilson Leibniz rule)}
\label{WLeibniz}
For every pair of complex functions $f$ and $g$ and every $n\in\mathbb{N}_0$, we have
\begin{align}
	\label{A15}\D_\W^n(fg) = \sum_{k=0}^n{C(n,k)\sum_{j=0}^{n-k}{\binom{n-k}{j}\A_\W^{n-k-j}\D_\W^{j+k}f\A_\W^j\D_\W^{n-j}g}},
\end{align}
where
\begin{align}
	\label{A16}C(n,k)=\left(-\frac14\right)^k\frac{(n-1+k)!}{(n-1-k)!k!}
\end{align}
for every pair of integers $0\le k\le n$.\ \ In \eqref{A16} we adopt the convention that $\frac{(-1)!}{(-1)!}:=1$ and $\frac{1}{(-1)!}:=0$.
\end{theorem}

We need the following two lemmas in proving Theorem~\ref{WLeibniz}.

\begin{lemma}
\label{commutator}
\[
	\A_\W\D_\W - \D_\W\A_\W = \frac{1}{2}\D_\W^2.
\]
\end{lemma}
\begin{proof}
	For every complex function $f$, we have
	\begin{align*}
		&\ \ \ \,((\A_\W\D_\W - \D_\W\A_\W)f)(x) \\
		&= \frac{1}{2}\left(\frac{f(x^{++})-f(x)}{2iz-1}+\frac{f(x)-f(x^{--})}{2iz+1}\right) \\
		&\ \ \ \,- \frac{1}{2iz}\left(\frac{f(x^{++})+f(x)}{2}-\frac{f(x)+f(x^{--})}{2}\right) \\
		&= \frac{1}{2}\left[(f(x^{++})-f(x))\left(\frac{1}{2iz - 1}-\frac{1}{2iz}\right) - (f(x)-f(x^{--}))\left(\frac{1}{2iz}-\frac{1}{2iz+1}\right)\right] \\
		&= \frac{1}{2}\frac{1}{2iz}\left(\frac{f(x^{++})-f(x)}{2iz - 1} - \frac{f(x)-f(x^{--})}{2iz + 1}\right) \\
		&= \frac{1}{2}(\D_\W^2 f)(x).
	\end{align*}
\end{proof}

\begin{lemma}
\label{Cnumrecur}
The numbers $C(n,k)$ as defined in \eqref{A16} satisfy $C(n,0)=1$ for all $n\in\mathbb{N}_0$, $C(n,n)=0$ for all $n\in\mathbb{N}$, and the recurrence
\begin{align}
	\label{Cnumbers} C(n,k) = \sum_{j=0}^k{\left(-\frac{1}{2}\right)^{k-j}\frac{(n-1-j)!}{(n-1-k)!}C(n-1,j)}
\end{align}
for every pair of integers $0<k<n$.
\end{lemma}

\begin{proof}
We first prove that the identity
\begin{align}
	\label{A17}\sum_{j=0}^k{\frac{1}{2^j}\frac{n-j}{n+j}\binom{n+j}{j}} = \frac{1}{2^k}\binom{n+k}{k}
\end{align}
holds for every $n\in\mathbb{N}$ and every $k\in\mathbb{N}_0$.\ \ To see this, we note that the left-hand side of \eqref{A17} can be written as a telescoping sum:
\begin{align*}
	\sum_{j=0}^k{\frac{1}{2^j}\frac{n-j}{n+j}\binom{n+j}{j}} &= \sum_{j=0}^k{\frac{1}{2^j}\left(1-\frac{2j}{n+j}\right)\binom{n+j}{j}} \\
	&= 1+\sum_{j=1}^k{\left[\frac{1}{2^j}\binom{n+j}{j}-\frac{1}{2^{j-1}}\binom{n+j-1}{j-1}\right]} \\
	&= \frac{1}{2^k}\binom{n+k}{k}.
\end{align*}
Now using \eqref{A17}, we have for every pair of integers $0<k<n$,
\begin{align*}
	&\ \ \ \,\sum_{j=0}^k{\left(-\frac{1}{2}\right)^{k-j}\frac{(n-1-j)!}{(n-1-k)!}C(n-1,j)} \\
	&=\sum_{j=0}^k{\left(-\frac{1}{2}\right)^{k-j}\frac{(n-1-j)!}{(n-1-k)!}\left(-\frac14\right)^j\frac{(n-2+j)!}{(n-2-j)!j!}} \\
	&=\frac{\left(-\frac{1}{4}\right)^{k}(n-1)!}{(n-1-k)!}\sum_{j=0}^k{2^{k-j}\frac{n-1-j}{n-1+j}\binom{n-1+j}{j}} \\
	&=\frac{\left(-\frac{1}{4}\right)^{k}(n-1)!}{(n-1-k)!}\binom{n-1+k}{k} =\left(-\frac14\right)^k\frac{(n-1+k)!}{(n-1-k)!k!}=C(n,k)
\end{align*}
which is \eqref{Cnumbers}.
\end{proof}

\noindent\textit{Proof of Theorem~\ref{WLeibniz}.}\ \ We prove \eqref{A15} by induction on $n$.\ \ \eqref{A15} is obviously true for $n=0$ and for $n=1$.\ \ Assuming that it is true for some $n\ge 1$ and noting that Lemma~\ref{commutator} implies that
\[
	\D_\W\A_\W^j = \sum_{l=0}^j{\binom{j}{l}\left(-\frac{1}{2}\right)^l l!\A_\W^{j-l}\D_\W^{l+1}}
\]
for every $j\in\mathbb{N}_0$, we also have
{\footnotesize
\begin{align*}
	&\ \ \ \,\D_\W^{n+1}(fg) \\
	&= \D_\W\D_\W^n(fg) \\
	&= \sum_{k=0}^n{C(n,k)\sum_{j=0}^{n-k}{\binom{n-k}{j}\D_\W(\A_\W^{n-k-j}\D_\W^{j+k}f\A_\W^j\D_\W^{n-j}g})} \\
	&= \sum_{k=0}^n{C(n,k)\sum_{j=0}^{n-k}{\binom{n-k}{j}\D_\W\A_\W^{n-k-j}\D_\W^{j+k}f\A_\W^{j+1}\D_\W^{n-j}g}} \\
	&\ \ \ \,+ \sum_{k=0}^n{C(n,k)\sum_{j=0}^{n-k}{\binom{n-k}{j}\A_\W^{n+1-k-j}\D_\W^{j+k}f\D_\W\A_\W^j\D_\W^{n-j}g}} \\
	&= \sum_{k=0}^n{C(n,k)\sum_{j=0}^{n-k}{\binom{n-k}{j}\left(\sum_{l=0}^{n-k-j}{\binom{n-k-j}{l}\left(-\frac{1}{2}\right)^l l! \A_\W^{n-k-j-l}\D_\W^{l+1}}\right)\D_\W^{j+k}f\A_\W^{j+1}\D_\W^{n-j}g}} \\
	&\ \ \ \,+ \sum_{k=0}^n{C(n,k)\sum_{j=0}^{n-k}{\binom{n-k}{j}\A_\W^{n+1-k-j}\D_\W^{j+k}f\left(\sum_{l=0}^j{\binom{j}{l}\left(-\frac{1}{2}\right)^l l!\A_\W^{j-l}\D_\W^{l+1}}\right)\D_\W^{n-j}g}} \\
	&= \sum_{k=0}^n{C(n,k)\sum_{p=1}^{n-k+1}\sum_{l=0}^{n-k-p+1}{\binom{n-k}{p-1}\binom{n-k-p+1}{l}\left(-\frac{1}{2}\right)^l l! \A_\W^{n-k-p-l+1}\D_\W^{p+k+l}f\A_\W^{p}\D_\W^{n-p+1}g}} \\
	&\ \ \ \,+ \sum_{k=0}^n{C(n,k)\sum_{j=0}^{n-k}\sum_{p=0}^j{\binom{n-k}{j}\binom{j}{j-p}\left(-\frac{1}{2}\right)^{j-p} (j-p)!\A_\W^{n+1-k-j}\D_\W^{j+k}f\A_\W^{p}\D_\W^{n-p+1}g}} \\
	&= \sum_{k=0}^n{C(n,k)\sum_{p=1}^{n-k+1}\sum_{j=p}^{n-k+1}{\binom{n-k}{p-1}\binom{n-k-p+1}{j-p}\left(-\frac{1}{2}\right)^{j-p} (j-p)! \A_\W^{n-k-j+1}\D_\W^{k+j}f\A_\W^{p}\D_\W^{n-p+1}g}} \\
	&\ \ \ \,+ \sum_{k=0}^n{C(n,k)\sum_{p=0}^{n-k}\sum_{j=p}^{n-k}{\binom{n-k}{j}\binom{j}{j-p}\left(-\frac{1}{2}\right)^{j-p} (j-p)!\A_\W^{n+1-k-j}\D_\W^{j+k}f\A_\W^{p}\D_\W^{n-p+1}g}}
\end{align*}
\begin{align*}
	&= \sum_{k=0}^n{C(n,k)\sum_{p=0}^{n-k+1}\sum_{j=p}^{n-k+1}{\binom{n-k}{p-1}\binom{n-k-p+1}{j-p}\left(-\frac{1}{2}\right)^{j-p} (j-p)! \A_\W^{n-k-j+1}\D_\W^{k+j}f\A_\W^{p}\D_\W^{n-p+1}g}} \\
	&\ \ \ \,+ \sum_{k=0}^n{C(n,k)\sum_{p=0}^{n-k+1}\sum_{j=p}^{n-k+1}{\binom{n-k}{j}\binom{j}{p}\left(-\frac{1}{2}\right)^{j-p} (j-p)!\A_\W^{n+1-k-j}\D_\W^{j+k}f\A_\W^{p}\D_\W^{n-p+1}g}}\\
	&= \sum_{k=0}^n{C(n,k)\sum_{p=0}^{n-k+1}\sum_{j=p}^{n-k+1}{\left(-\frac{1}{2}\right)^{j-p}\left[\binom{n-k}{p-1}\binom{n-k-p+1}{j-p} (j-p)! +\binom{n-k}{j}\binom{j}{p} (j-p)!\right]}} \\
	&\ \ \ \,\A_\W^{n+1-k-j}\D_\W^{j+k}f\A_\W^{p}\D_\W^{n-p+1}g \\
	&= \sum_{k=0}^n{C(n,k)\sum_{p=0}^{n-k+1}\sum_{j=p}^{n-k+1}{\left(-\frac{1}{2}\right)^{j-p}\frac{(n-k)!}{(n-k-j+p)!}\binom{n-k-j+p+1}{p}\A_\W^{n+1-k-j}\D_\W^{j+k}f\A_\W^{p}\D_\W^{n-p+1}g}} \\
	&= \sum_{k=0}^n{C(n,k)\sum_{p=0}^{n-k+1}\sum_{m=k}^{n-p+1}{\left(-\frac{1}{2}\right)^{m-k}\frac{(n-k)!}{(n-m)!}\binom{n-m+1}{p}\A_\W^{n+1-m-p}\D_\W^{m+p}f\A_\W^{p}\D_\W^{n-p+1}g}} \\
	&= \sum_{m=0}^{n+1}{\sum_{k=0}^{m}\sum_{p=0}^{n-m+1}{\left(-\frac{1}{2}\right)^{m-k}\frac{(n-k)!}{(n-m)!}C(n,k)\binom{n-m+1}{p}\A_\W^{n+1-m-p}\D_\W^{m+p}f\A_\W^{p}\D_\W^{n-p+1}g}} \\
	&= \sum_{m=0}^{n+1}{C(n+1,m)\sum_{p=0}^{n+1-m}{\binom{n+1-m}{p}\A_\W^{n+1-m-p}\D_\W^{p+m}f\A_\W^p\D_\W^{n+1-p}g}},
\end{align*}}%
where the last step follows from Lemma~\ref{Cnumrecur}.
\hfill\qed \\

There is a simpler recurrence formula generating the numbers $C(n,k)$ than \eqref{Cnumbers}, which is
\[
	C(n,k) = C(n-1,k) - \frac{n+k-2}{2}C(n-1,k-1)
\]
for every pair of integers $0<k<n$, in which only three terms in the array are involved.\ \ In fact, $C(n,k)$ is $(-\frac12)^k$ times the coefficient of $x^k$ in the Bessel polynomial $y_{n-1}$ of degree $n-1$ \cite{Grosswald}.\ \ Table 1 shows the numbers $C(n,k)$ for small values of $n$ and $k$.

\begin{table}[htp]
\begin{center}
\begin{tabular}{|c||c|c|c|c|c|c|}
\hline
\rule{0px}{10px}$C(n,k)$ & $k=0$ & $k=1$ & $k=2$ & $k=3$ & $k=4$ & $k=5$ \\
\hline
\hline
\rule{0px}{10px}$n=0$ & $1$ & & & & & \\
\hline
\rule{0px}{10px}$n=1$ & $1$ & $0$ & & & & \\
\hline
\rule{0px}{10px}$n=2$ & $1$ & $-\frac{1}{2}$ & $0$ & & & \\
\hline
\rule{0px}{10px}$n=3$ & $1$ & $-\frac{3}{2}$ & $\frac{3}{4}$ & $0$ & & \\
\hline
\rule{0px}{10px}$n=4$ & $1$ & $-3$ & $\frac{15}{4}$ & $-\frac{15}{8}$ & $0$ & \\
\hline
\rule{0px}{10px}$n=5$ & $1$ & $-5$ & $\frac{45}{4}$ & $-\frac{105}{8}$ & $\frac{105}{16}$ & $0$ \\
\hline
\end{tabular}
\end{center}
\caption{Values of $C(n,k)$ for small $n$ and $k$}
\end{table}

\pagebreak
Now we turn to some analytic properties of the Wilson operator.\ \ We can check easily from the definition of the Wilson operator $\D_\W$ that it sends polynomials to polynomials.\ \ In fact, we have the following.
\begin{proposition}
\label{Mero}
\cite[Proposition 2.3]{Cheng-Chiang}
	Let $f$ be a complex function.\ \ Then
	\begin{enumerate}[(i)]
	\item if $f$ is entire, then $\D_\W f$ and $\A_\W f$ are also entire;
	\item if $f$ is meromorphic, then $\D_\W f$ is also meromorphic; and
	\item if $f$ is rational, then $\D_\W f$ is also rational.
	\end{enumerate}
\end{proposition}

\bigskip
After introducing the Wilson operator and some of its useful properties, we will turn our focus to a series expansion of entire functions in a polynomial basis based on the Wilson operator.
\begin{definition}
\label{Wbasis}
	Let $x_0\in\mathbb{C}$.\ \ Then the \textbf{\textit{Wilson basis}} $\{\tau_k(x;x_0):k\in\mathbb{N}_0\}$ is defined as $\tau_0(x;x_0):=1$, and
	\[
		\tau_k(x;x_0) := \prod_{j=0}^{k-1}{(x_0^{+(2j)}-x)} = \prod_{j=0}^{k-1}{[(z_0+ji)^2-x]}
	\]
for $k\in\mathbb{N}$, where $z_0:=\sqrt{x_0}$ with the branch of square root $\arg{z_0}\in(-\frac{\pi}{2},\frac{\pi}{2}]$ if $x_0\ne 0$.\ \ We emphasize that whenever the Wilson basis $\{\tau_k(x;x_0):k\in\mathbb{N}_0\}$ is mentioned, the symbol $z_0$ automatically takes the aforesaid meaning.
\end{definition}

Note that the choice of branch of square root in Definition~\ref{Wbasis} ensures that $2z_0i\notin\mathbb{N}$, so that $x_0,x_0^{++},x_0^{+(4)},\ldots$ are distinct points in $\mathbb{C}$.

\bigskip
The Wilson operator interacts with the Wilson basis in a similar way as the ordinary differential operator $\frac{d}{dx}$ does with the basis $\{(x-x_0)^k:k\in\mathbb{N}_0\}$.\ \ The following formula is the Wilson counterpart of the ordinary differentiation formula $\frac{d}{dx}x^k=kx^{k-1}$ for every positive integer $k$.\ \ Its $q$-analogue was discussed in \cite{Askey-Wilson}.
\begin{proposition}
\label{TauRule}
\cite[Proposition 1.9]{Cheng}
	For every $x_0\in\mathbb{C}$ and every $k\in\mathbb{N}$, we have
	\begin{align*}
		(\D_\W\tau_k)(x;x_0) = -k\tau_{k-1}(x;x_0^+).
	\end{align*}
\end{proposition}

The following theorem from Gelfond's book \cite{Gelfond} implies that a Wilson series either converges uniformly on every compact subset of $\mathbb{C}$, or converges nowhere except at the points $x_0^{+(2)},x_0^{+(4)},x_0^{+(6)},\ldots$ where the series terminates.\ \ In other words, the domain of a function defined by a Wilson series is either ``all" (the whole $\mathbb{C}$) or ``nothing" (just some isolated points).
\begin{theorem}
\label{convergence0}
\cite[p. 172]{Gelfond} Let $\{a_n\}_{n\in\mathbb{N}_0}$ and $\{x_n\}_{n\in\mathbb{N}_0}$ be sequences of complex numbers such that
\[
	\sum_{k=0}^\infty\frac{1}{|x_k|}<+\infty.
\]
If the polynomial series
\[
	\sum_{k=0}^\infty{a_k(x-x_0)\cdots(x-x_{k-1})}
\]
converges at a point $a\in\mathbb{C}\setminus\{x_0,x_1,x_2,\ldots\}$, then it converges uniformly on every compact subset of $\mathbb{C}$.\ \ In particular, given $x_0\in\mathbb{C}$ and a sequence $\{a_n\}_{n\in\mathbb{N}_0}$ of complex numbers, the Wilson series
\[
	\sum_{k=0}^\infty{a_k\tau_k(x;x_0)}
\]
either converges nowhere except at the points $x_0,x_0^{++},x_0^{+(4)},\ldots$ or converges uniformly on every compact subset of $\mathbb{C}$.
\end{theorem}

\bigskip
We next investigate the Wilson series expansion of an entire function.\ \ We will find out the coefficients in the expansion and look for a condition for uniform convergence of such a series on compact subsets of $\mathbb{C}$.
\begin{theorem}
\label{WSeries}
\textup{(Wilson series expansion)}
Let $x_0\in\mathbb{C}$ and let $f$ be an entire function satisfying
\begin{align}
	\label{A7} \limsup_{r\to\infty}{\frac{\ln^+ M(r;f)}{\sqrt{r}}} < 2\ln 2.
\end{align}
Then there exists a unique sequence of complex numbers $\{a_n\}_{n\in\mathbb{N}_0}$, given by
\begin{align}
	\label{A13} a_n=\frac{1}{n!}\sum_{j=0}^n{(-1)^{n-j}\binom{n}{j}\frac{1}{(-2z_0i+j)_j(-2z_0i+2j+1)_{n-j}}f(x_0^{+(2j)})},
\end{align}
such that the Wilson series
\[
	\sum_{k=0}^\infty{a_k\tau_k(x;x_0)}
\]
converges uniformly to $f$ on every compact subset of $\mathbb{C}$.
\end{theorem}

\begin{proof}
	For every pair of distinct points $x,y\in\mathbb{C}$, we have an expansion of the Cauchy kernel
	\[
		\frac{1}{y-x} = \frac{1}{y-x}\frac{\tau_{n+1}(x;x_0)}{\tau_{n+1}(y;x_0)} - \sum_{k=0}^n{\frac{\tau_k(x;x_0)}{\tau_{k+1}(y;x_0)}}
	\]
	for every $x_0\in\mathbb{C}$ and every $n\in\mathbb{N}$.
	
	Now given any compact set $K\subset\mathbb{C}$ and any $x\in K$, there exists $N\in\mathbb{N}$ such that the disk $D(0;4n^2)$ contains all the points $x,x_0,x_0^{++},x_0^{+(4)},\ldots,x_0^{+(2n)}$ for every $n\ge N$.\ \ So by Cauchy integral formula we have
	\begin{align}
	\label{A8}
	\begin{aligned}
		f(x) &= \frac{1}{2\pi i}\int_{\partial D(0;4n^2)}{\frac{f(y)}{y-x}\,dy} \\
		&= \frac{1}{2\pi i}\int_{\partial D(0;4n^2)}{\frac{f(y)\tau_{n+1}(x;x_0)}{(y-x)\tau_{n+1}(y;x_0)}\,dy} + \sum_{k=0}^n{a_k\tau_k(x;x_0)},
	\end{aligned}
	\end{align}
	where the coefficients $a_k$ are given, using Cauchy's residue theorem, by
	\begin{align*}
		a_k &= -\frac{1}{2\pi i}\int_{\partial D(0;4n^2)}{\frac{f(y)}{\tau_{k+1}(y;x_0)}\,dy} = \sum_{j=0}^k{\frac{f(x_0^{+(2j)})}{\prod_{\substack{l=0\\l\ne j}}^k{[(z_0+li)^2-(z_0+ji)^2]}}} \\
		&= \sum_{j=0}^k{\frac{(-1)^j}{j!(k-j)!}\frac{1}{\prod_{l=0}^{j-1}{[2z_0i-(j+l)]}\prod_{l=j+1}^k{[2z_0i-(j+l)]}}f(x_0^{+(2j)})} \\
		&= \frac{1}{k!}\sum_{j=0}^k{(-1)^{k-j}\binom{k}{j}\frac{1}{(-2z_0i+j)_j(-2z_0i+2j+1)_{k-j}}f(x_0^{+(2j)})}.
	\end{align*}
	
	It remains to show that the integral on the right-hand side of \eqref{A8} converges uniformly to $0$ on $K$ as $n\to\infty$.\ \ Note that for every $y\in\partial D(0;4n^2)$, we have $|z_0\pm\sqrt{y}|>n$ for every sufficiently large $n$.\ \ This implies that
	\begin{align}
	\label{A60}
	\begin{aligned}
		&\ \ \ \,\left|\frac{\tau_{n+1}(x;x_0)}{\tau_{n+1}(y;x_0)}\right| \\
		&= \left|\prod_{j=0}^n{\frac{(z_0+ji)^2-x}{(z_0+ji)^2-y}}\right| = \left|\prod_{j=0}^n{\frac{(z_0+ji-\sqrt{x})(z_0+ji+\sqrt{x})}{(z_0+ji-\sqrt{y})(z_0+ji+\sqrt{y})}}\right| \\
		&\le \prod_{j=0}^n{\frac{(|z_0-\sqrt{x}|+j)(|z_0+\sqrt{x}|+j)}{(|z_0-\sqrt{y}|-j)(|z_0+\sqrt{y}|-j)}} \\
		&= \frac{\Gamma(|z_0-\sqrt{x}|+n+1)\Gamma(|z_0+\sqrt{x}|+n+1)}{\Gamma(|z_0-\sqrt{x}|)\Gamma(|z_0+\sqrt{x}|)}\frac{\Gamma(|z_0-\sqrt{y}|-n)\Gamma(|z_0+\sqrt{y}|-n)}{\Gamma(|z_0-\sqrt{y}|+1)\Gamma(|z_0+\sqrt{y}|+1)} \\
		&= O\left(\frac{(n^{|z_0-\sqrt{x}|+n+\frac{1}{2}}e^{-n})(n^{|z_0+\sqrt{x}|+n+\frac{1}{2}}e^{-n})(n^{n+\frac{1}{2}}e^{-n})(n^{n+\frac{1}{2}}e^{-n})}{[(2n)^{2n-\frac{1}{2}}e^{-2n}][(2n)^{2n-\frac{1}{2}}e^{-2n}]}\right) \\
		&= O\left(\frac{n^{4\sqrt{N}+3}}{2^{4n}}\right)
	\end{aligned}
	\end{align}
	as $n\to\infty$, where in the second last step we have applied Stirling's approximation along the positive real axis.

	Now if \eqref{A7} holds, then there exists $\varepsilon\in(0,2\ln 2)$ such that $\displaystyle\frac{\ln^+ M(r;f)}{\sqrt{r}}\le2\ln 2-\varepsilon$ for every sufficiently large $r$, which implies that
	\begin{align}
	\label{A61}
		M(4n^2;f) \le \frac{2^{4n}}{e^{2\varepsilon n}}
	\end{align}
	for every sufficiently large $n$.\ \ \eqref{A60} and \eqref{A61} imply that the integral on the right-hand side of \eqref{A8} converges uniformly to $0$ on $K$ as $n\to\infty$.
	
	To show the uniqueness statement, we suppose that $\displaystyle\sum_{k=0}^\infty{a_k\tau_k(x;x_0)}$ converges uniformly to $0$ on compact subsets of $\mathbb{C}$ and aim to show that $a_n=0$ for every $n\in\mathbb{N}_0$.\ \ This follows easily by successively applying $\D_\W$ on both sides of $\displaystyle\sum_{k=0}^\infty{a_k\tau_k(x;x_0)}=0$.
\end{proof}

The uniqueness statement in Theorem~\ref{WSeries}, together with the fact that
\[
	a_n=\frac{(-1)^n}{n!}(\D_\W^n f)(x_0^{+(n)}),
\]
implies that a Wilson series can only represent functions that satisfy \eqref{A7}.\ \ Thus if $f$ is a non-constant entire function in $\ker\D_\W$, then
\[
	\limsup_{r\to\infty}{\frac{\ln M(r;f)}{\sqrt{r}}} \ge 2\ln 2,
\]
i.e. $f$ is of order at least $\frac{1}{2}$, and of type at least $2\ln 2$ in case the order is exactly $\frac{1}{2}$.

\bigskip
The following corollary relates the $n$th Wilson difference of $f$ at a point and the values taken by $f$ at the nearby interpolation points.\ \ The formula was first introduced by Cooper \cite{Cooper}.

\begin{corollary}
\label{WilsonExp}
Let $f$ be an entire function satisfying \eqref{A7}.\ \ Then at each point $x_0\in\mathbb{C}$, we have
\[
	(\D_\W^n f)(x_0^{+(n)}) = \sum_{j=0}^n{(-1)^j\binom{n}{j}\frac{1}{(-2z_0i+j)_j(-2z_0i+2j+1)_{n-j}}f(x_0^{+(2j)})}
\]
for every non-negative integer $n$.\ \ Replacing $z_0$ by $(z_0-\frac{ni}{2})$, we also have
\[
	(\D_\W^n f)(x_0) = (-1)^n\sum_{j=0}^n{\binom{n}{j}\frac{1}{(-2z_0i-n+j)_j(2z_0i-j)_{n-j}}f(x_0^{+(2j-n)})}
\]
for every non-negative integer $n$.
\end{corollary}


\begin{proposition}
\label{C}
For every pair of non-negative integers $k$ and $n$, we let
\[
	T(k,n):= \frac{(-1)^{n+k}}{n!}\left.\D_\W^n x^k\right|_{x=0^{+(n)}}.
\]
Then $T(0,0)=1$, $T(k,0)=0$ for all $k\in\mathbb{N}$, $T(k,n)=0$ for all non-negative integers $k$ and $n$ with $k<n$, and
\begin{align}
	\label{AA} T(k,n)=\sum_{j=1}^n{\frac{2(-1)^{n+j}j^{2k}}{(n-j)!(n+j)!}}
\end{align}
for every pair of positive integers $k$ and $n$.\ \ In particular, we have the following:
\begin{enumerate}[(i)]
\item $T(k,n)$ satisfy the recurrence
\[
	T(k,n) = T(k-1,n-1) + n^2 T(k-1,n)
\]
for every pair of positive integers $k$ and $n$.\ \ In particular, $T(k,n)\ge 0$ for every pair of non-negative integers $k$ and $n$.
\item There exists $K>0$ such that for every pair of positive integers $k$ and $n$, we have
\[
	T(k,n) \le \frac{Ke^{2n}}{n^{2n}}n^{2k}.
\]
\end{enumerate}
\end{proposition}
\begin{proof}
(i) follows immediately from \eqref{AA}, which is a direct consequence of Corollary~\ref{WilsonExp} applied to the function $f(x)=x^k$ and the point $x_0=0$.\ \ To prove (ii), we note that \eqref{AA} implies that
\begin{align*}
	\frac{T(k,n)}{\frac{2}{(2n)!}n^{2k}} &= \sum_{j=1}^n{(-1)^{n+j}\binom{2n}{n+j}\left(\frac{j}{n}\right)^{2k}} \le \sum_{j=1}^n{\binom{2n}{n+j}} \le 2^{2n}
\end{align*}
for every pair of positive integers $k$ and $n$, so the estimate follows from Stirling's approximation.
\end{proof}

The numbers $T(k,n)$ in Proposition~\ref{C} are called the Carlitz-Riordan \textit{central factorial numbers} or the \textit{Chebyshev-Stirling numbers of the second kind}.\ \ In fact $T(k,n)$ is the number of partitions of the set $\{1,1',2,2',...,k,k'\}$ into $n$ disjoint nonempty subsets $V_1,\ldots,V_n$ such that, for each $1\le j\le n$, if $i$ is the smallest integer such that either $i\in V_j$ or $i'\in V_j$ then $\{i,i'\}\subseteq V_j$.\ \ Therefore $T(k,n)$ can be regarded as a ``two-colored" version of the Stirling numbers of the second kind.\ \ These numbers are first investigated by Carlitz and Riordan in \cite{CR}, and their properties and asymptotics are studied in \cite{Charalambides}, \cite{GLN}, \cite{BSSV}, etc.\ \ Matsumoto and Novak \cite{MN} gave another combinatorial interpretation for these numbers, which is the number of primitive factorizations of a full cycle.

\bigskip
After studying the Askey-Wilson series expansion of entire functions, we will develop a Wiman-Valiron theory for this series expansion in the following sections.

\section{Main results}
\label{sec:Mainresults}

We have seen in Theorem~\ref{WSeries} that an entire function of order smaller than $\frac{1}{2}$ has a Wilson series expansion at the point $x_0=0$, which converges uniformly to itself on compact subsets of $\mathbb{C}$.\ \ Such a Wilson series expansion must converge at each positive real number $r$ in particular, so we are able to make the following definition.

\begin{definition}
\label{Wilsonmunu}
Let $f\not\equiv0$ be an entire function of order smaller than $\frac{1}{2}$ with Wilson series expansion $\displaystyle f(x)=\sum_{k=0}^\infty{a_k\tau_k(x;0)}$.\ \ The \textbf{\textit{Wilson maximal term}} and \textbf{\textit{Wilson central index}} of $f$ are respectively the functions $\mu_\W(\cdot;f):(0,+\infty)\to(0,+\infty)$ and $\nu_\W(\cdot;f):(0,+\infty)\to\mathbb{N}_0$ defined by
\begin{align*}
	\mu_\W(r;f)&:= \max_{n\in\mathbb{N}_0}{\max_{x\in\partial D(0;r)}|a_n\tau_n(x;0)|} =\max_{n\in\mathbb{N}_0}{|a_n|r(r+1^2)\cdots(r+(n-1)^2)}
\end{align*}
and
\[
	\nu_\W(r;f):= \max\{n\in\mathbb{N}_0:|a_n|r(r+1^2)\cdots(r+(n-1)^2)=\mu_\W(r;f)\}.
\]
\end{definition}

Here we only focus on Wilson series expansions at $x_0=0$ for simplicity.\ \ If we consider Wilson series expansions at any $x_0\in\mathbb{C}$, then
\[
	\mu_\W(r;f):= \max_{n\in\mathbb{N}_0}{\max_{x\in\partial D(x_0;r)}|a_n\tau_n(x;x_0)|}
\]
will lie between $\displaystyle\max_{n}|a_n|\prod_{k=0}^{n-1}(r+k^2-2k|z_0|)$ and $\displaystyle\max_{n}|a_n|\prod_{k=0}^{n-1}(r+k^2+2k|z_0|)$ for every $r>0$, and the analysis will be similar.

\bigskip
The following is the main theorem of this paper, which is about an entire function of order smaller than $\frac13$.\ \ In the case $h=0$, it says that in the Wilson series expansion of such an entire function, the terms that are far away from the maximal term are small outside a small exceptional set.\ \ In other words, the local behavior of such an entire function is mainly contributed by those terms in its Wilson series expansion that are near the maximal term.

\begin{theorem}
\label{tail}
Let $\displaystyle f(x)=\sum_{k=0}^\infty{a_k \tau_k(x;0)}$ be a transcendental entire function of order $\sigma<\frac13$, let $\gamma\in(3,\frac{1}{\sigma})$, and let $\delta>0$.\ \ Then there exists a set $E\subset[1,\infty)$ of finite logarithmic measure\footnote{The \textbf{\textit{logarithmic measure}} of a set $E\subseteq[1,\infty)$ is defined by $\logmea E:=\int_E\frac{1}{x}\,dm$, where $m$ is the Lebesgue measure on $\mathbb{R}$.} such that for every $h\in\mathbb{R}$, $\beta>0$ and $\omega\in(0,\beta)$, we have
\[
	\sum_{k:|k-N|\ge \kappa}{k^h|a_k\tau_k(r;0)|} = o(\mu_\W(r;f)N^h b(N)^{\frac{\omega-1}{2}})
\]
as $r\to\infty$ and $r\in(0,\infty)\setminus E$, where $N=\nu_\W(r;f)$, $b(N):=\frac{1}{N\ln N (\ln\ln N)^{1+\delta}}$ and $\kappa=\left[\sqrt{\frac{\beta}{b(N)}\ln\frac{1}{b(N)}}\right]$.
\end{theorem}

Applying Theorem~\ref{tail}, we obtain the following asymptotic behavior of successive Wilson differences of a transcendental entire function of order smaller than $\frac13$.

\begin{theorem}
\label{WVmain}
Let $\displaystyle f(x)=\sum_{k=0}^\infty{a_k \tau_k(x;0)}$ be a transcendental entire function of order $\sigma<\frac13$, let $\gamma\in(3,\frac{1}{\sigma})$, and let $\delta>0$.\ \ Then there exists a set $E\subset[1,\infty)$ of finite logarithmic measure such that for every $n\in\mathbb{N}$, we have
\[
	\left(\frac{x}{N}\right)^n(\D_\W^n f)(x) = f(x) + O\left(\frac{\kappa}{N}\right)M(r;f)
\]
as $r\to\infty$ and $r\in(0,\infty)\setminus E$, where in the above formula $r=|x|$, $N=\nu_\W(r;f)$ and $\kappa=\left[\sqrt{N(\ln N)^2(\ln\ln N)^{1+\delta}}\right]$.
\end{theorem}

\section{Properties of the Wilson maximal term and central index}
\label{sec:Propmunu}

We start by stating three lemmas which are about some useful properties of the functions $\mu_\W(\cdot;f)$ and $\nu_\W(\cdot;f)$.\ \ All the results in this section hold for entire functions of order smaller than $\frac12$.

\begin{lemma}
\label{munuprop}
Let $f$ be a non-constant entire function of order smaller than $\frac{1}{2}$.\ \ Then $\mu_\W(\cdot;f)$ and $\nu_\W(\cdot;f)$ have the following properties.
\begin{enumerate}[(i)]
	\item $\nu_\W(\cdot;f)$ is a right-continuous non-decreasing piecewise-constant function.\ \ If $f$ is transcendental, then $\displaystyle\lim_{r\to\infty}{\nu_\W(r;f)}=+\infty$.
	\item $\mu_\W(\cdot;f)$ is continuous everywhere and strictly increasing on $[R,+\infty)$ for some $R>0$.\ \ Also $\displaystyle\lim_{r\to\infty}{\mu_\W(r;f)}=+\infty$.
\end{enumerate}
\end{lemma}
\begin{proof}
We write $\displaystyle f(x)=\sum_{k=0}^\infty{a_k\tau_k(x;0)}$.
\begin{enumerate}[(i)]
	\item We first prove that $\nu_\W(\cdot;f)$ is non-decreasing.\ \ Let $r_1>r_2>0$ and denote $N_1:=\nu_\W(r_1;f)$ and $N_2:=\nu_\W(r_2;f)$.\ \ If on the contrary $N_1<N_2$, then
	\begin{align*}
		\frac{|a_{N_2}|r_2(r_2+1^2)\cdots(r_2+(N_2-1)^2)}{|a_{N_1}|r_2(r_2+1^2)\cdots(r_2+(N_1-1)^2)} &= \frac{|a_{N_2}|}{|a_{N_1}|}(r_2+N_1^2)\cdots(r_2+(N_2-1)^2) \\
		&< \frac{|a_{N_2}|}{|a_{N_1}|}(r_1+N_1^2)\cdots(r_1+(N_2-1)^2) \\
		&= \frac{|a_{N_2}|r_1(r_1+1^2)\cdots(r_1+(N_2-1)^2)}{|a_{N_1}|r_1(r_1+1^2)\cdots(r_1+(N_1-1)^2)} \\
		&\le 1,
	\end{align*}
	which is a contradiction.\ \ So $\nu_\W(\cdot;f)$ is non-decreasing.
	
	Next we prove that $\nu_\W(\cdot;f)$ is right-continuous and piecewise constant.\ \ Let $r_0>0$ and denote $N_0:=\nu_\W(r_0;f)$.\ \ Since the Wilson series of $f$ converges absolutely at $2r_0$, we have
	\[
		\lim_{k\to\infty}{|a_k|(2r_0)(2r_0+1^2)\cdots(2r_0+(k-1)^2)}=0,
	\]
	and so there exists $K>N_0$ such that
	\begin{align}
		\label{W3}|a_k|(2r_0)(2r_0+1^2)\cdots(2r_0+(k-1)^2)<\mu_\W(r_0;f)
	\end{align}
	for every $k>K$.\ \ Also, by the definition of $N_0$ we must have
	\begin{align}
		\label{W4}|a_k|r_0(r_0+1^2)\cdots(r_0+(k-1)^2)<\mu_\W(r_0;f)
	\end{align}
	for every $k>N_0$.\ \ Since each polynomial $|a_k|r(r+1^2)\cdots(r+(k-1)^2)$ is continuous, \eqref{W4} implies that there exists $\delta\in(0,r_0)$ such that for every $r\in[r_0,r_0+\delta]$ and for those (finitely many) $k\in\{N_0+1,\ldots,K\}$ we have
	\begin{align}
		\label{W5}|a_k|r(r+1^2)\cdots(r+(k-1)^2)<\mu_\W(r_0;f).
	\end{align}
	On the other hand, for every $r\in[r_0,r_0+\delta]\subseteq[r_0,2r_0]$ and every $k>K$, \eqref{W3} gives
	\begin{align}
		\nonumber |a_k|r(r+1^2)\cdots(r+(k-1)^2)&\le|a_k|(2r_0)(2r_0+1^2)\cdots(2r_0+(k-1)^2) \\
		\label{W6} &<\mu_\W(r_0;f).
	\end{align}
	So combining \eqref{W5} and \eqref{W6}, for every $r\in[r_0,r_0+\delta]$ and every $k>N_0$ we have
	\begin{align*}
		|a_k|r(r+1^2)\cdots(r+(k-1)^2)&<\mu_\W(r_0;f) \\
		&=|a_{N_0}|r_0(r_0+1^2)\cdots(r_0+(N_0-1)^2),
	\end{align*}
	which implies that $\nu_\W(r;f)\le N_0$ for every $r\in[r_0,r_0+\delta]$.\ \ But since $\nu_\W(r;f)$ is non-decreasing, we must have $\nu_\W(r;f)=N_0$ for every $r\in[r_0,r_0+\delta]$.\ \ Since $r_0\in(0,+\infty)$ was arbitrary, it follows that $\nu_\W(\cdot;f)$ is right-continuous and piecewise constant.
	
	Finally we let $\displaystyle a:=\max_{n\in\mathbb{N}_0}{|a_n|}$.\ \ Then for every $n\in\mathbb{N}_0$ and every $r>0$ we have
	\[
		|a_n|r(r+1^2)\cdots(r+(n-1)^2) \le ar(r+1^2)\cdots(r+(\nu_\W(r;f)-1)^2).
	\]
	So for all those $n\in\mathbb{N}_0$ with $a_n\ne 0$ we have
	\begin{align}
		\label{W7} n\le\liminf_{r\to\infty}{\nu_\W(r;f)}.
	\end{align}
	If $f$ is transcendental, then there are infinitely many $n\in\mathbb{N}_0$ with $a_n\ne 0$.\ \ \eqref{W7} now holds for infinitely many $n$, so we have $\displaystyle\lim_{r\to\infty}{\nu_\W(r;f)}=+\infty$.
	\item We first prove that $\mu_\W(\cdot;f)$ is continuous on $(0,+\infty)$.\ \ Since $f$ is non-constant, there exists $n\in\mathbb{N}$ such that $a_n\ne 0$.\ \ Now let $r_0\in(0,+\infty)$.\ \ Similar to (i), there exists $K>n$ such that
	\[
		|a_k|(2r_0)(2r_0+1^2)\cdots(2r_0+(k-1)^2)<|a_n|r_0(r_0+1^2)\cdots(r_0+(n-1)^2)
	\]
	for every $k>K$.\ \ Then for every $r\in[r_0,2r_0]$ we have
	\begin{align*}
		|a_k|r(r+1^2)\cdots(r+(k-1)^2)&\le|a_k|(2r_0)(2r_0+1^2)\cdots(2r_0+(k-1)^2) \\
		&<|a_n|r_0(r_0+1^2)\cdots(r_0+(n-1)^2) \\
		&\le|a_n|r(r+1^2)\cdots(r+(n-1)^2)
	\end{align*}
	for every $k>K$, so
	\[
		\mu_\W(r;f)=\max_{k\in\{0,1,\ldots,K\}}{|a_k|r(r+1^2)\cdots(r+(k-1)^2)}.
	\]
	Being the maximum of finitely many polynomials on $[r_0,2r_0]$, $\mu_\W(\cdot;f)$ is continuous on $[r_0,2r_0]$.\ \ Since $r_0\in(0,+\infty)$ was arbitrary, $\mu_\W(\cdot;f)$ is continuous on $(0,+\infty)$.
	
	Next, since $f$ is non-constant and $\nu_\W(\cdot;f)$ is non-decreasing, there exists $R>0$ such that $\nu_\W(r;f)\ge 1$ for every $r>R$.\ \ Now for every $r_1>r_2\ge R$, denoting $N:=\nu_\W(r_2;f)$, we have $|a_N|\ne 0$ and
	\begin{align*}
		\mu_\W(r_2;f) &= |a_N|r_2(r_2+1^2)\cdots(r_2+(N-1)^2) \\
		&< |a_N|r_1(r_1+1^2)\cdots(r_1+(N-1)^2) \\
		&\le \mu_\W(r_1;f),
	\end{align*}
	so $\mu_\W(\cdot;f)$ is strictly increasing on $[R,+\infty)$.

	Finally, since $f$ is non-constant, there exists $n\in\mathbb{N}$ such that $a_n\ne 0$.\ \ Thus we have
	\[
		\mu_\W(r;f) \ge |a_n|r(r+1^2)\cdots(r+(n-1)^2)
	\]
	for every $r\in[0,+\infty)$, and so $\displaystyle\lim_{r\to\infty}{\mu_\W(r;f)}=+\infty$.
\end{enumerate}
\end{proof}

\begin{lemma}
\label{gg}
Let $\displaystyle f(x)=\sum_{k=0}^\infty{a_k\tau_k(x;0)}$ be an entire function of order smaller than $\frac{1}{2}$, and let $\gamma>3$.\ \ Then for each $n\in\mathbb{N}_0$, there exists $K_n>1$ such that
	\[
		|a_n\tau_n(r;0)| \le K_n M(r;f)
	\]
for every $r>\max\{4n^2,n^\gamma\}$, and the sequence $\{K_n\}_{n\in\mathbb{N}_0}$ decreases to $1$.
\end{lemma}
\begin{proof}
We let $K_0=K_1=9$ and $K_n:=\left(1+\frac{1}{n^{\gamma-2}}\right)^n\left(1-\frac{1}{n^{\gamma-2}}\right)^{-n}$ for $n\ge 2$.\ \ Then the sequence $\{K_n\}_{n\in\mathbb{N}_0}$ decreases to $1$.\ \ Now for each $n\in\mathbb{N}_0$ and each $r>\max\{4n^2,n^\gamma\}$, applying Cauchy's Residue Theorem as in Theorem~\ref{WSeries}, we have
\[
	a_n = -\frac{1}{2\pi i}\int_{\partial D(0;r)}{\frac{f(y)}{\tau_{n+1}(y;0)}\,dy},
\]
so
\begin{align*}
	|a_n \tau_n(r;0)| &\le \frac{1}{2\pi}\frac{2\pi r M(r;f)}{r(r-1^2)(r-2^2)\cdots(r-n^2)}r(r+1^2)\cdots(r+(n-1)^2) \\
	&= \frac{(1+\frac{0^2}{r})(1+\frac{1^2}{r})\cdots(1+\frac{(n-1)^2}{r})}{(1-\frac{1^2}{r})(1-\frac{2^2}{r})\cdots(1-\frac{n^2}{r})}M(r;f) \\
	&\le \left(1+\frac{n^2}{r}\right)^n\left(1-\frac{n^2}{r}\right)^{-n} M(r;f) \\
	&\le K_n M(r;f).
\end{align*}
\end{proof}

\begin{lemma}
\label{muorder}
Let $f\not\equiv0$ be an entire function of order $\sigma<\frac{1}{2}$.\ \ Then
\begin{enumerate}[(i)]
	\item $\displaystyle\sigma_{\mu_\W(\cdot;f)}=\limsup_{r\to\infty}\frac{\ln^+\nu_\W(r;f)}{\ln r}\le\sigma$.
	\item In particular, for every $\gamma<\frac{1}{\sigma_{\mu_\W(\cdot;f)}}$, we have
	\[
		\nu_\W(r;f)^\gamma\le r
	\]
	for every sufficiently large $r\in(0,+\infty)$.
\end{enumerate}
\end{lemma}
The inequality $\sigma_{\mu_\W(\cdot;f)}\le\sigma$ in Lemma~\ref{muorder} (i) can be improved to an equality for an entire function of order $\sigma<\frac13$, but for $\sigma\in[\frac13,\frac12)$ we are currently unsure about whether this can be improved.\ \ We state this result as the following theorem, but delay its proof to \S\ref{sec:Proofs} and only prove Lemma~\ref{muorder} in the meantime.

\begin{theorem}
\label{muorderequal}
Let $f\not\equiv0$ be an entire function of order $\sigma<\frac13$.\ \ Then
\[
	\sigma=\sigma_{\mu_\W(\cdot;f)}=\limsup_{r\to\infty}\frac{\ln^+\nu_\W(r;f)}{\ln r}.
\]
\end{theorem}

\bigskip
\noindent \textit{Proof of Lemma~\ref{muorder}.}\ \ Let the Maclaurin series expansion of $f$ be $\displaystyle f(x)=\sum_{k=0}^\infty{b_k x^k}$.\ \ Since $\sigma<\frac{1}{2}$, Theorem~\ref{WSeries} implies that there exists a sequence $\{a_n\}_{n\in\mathbb{N}_0}$ of complex numbers such that
\[
	f(x)=\sum_{k=0}^\infty{a_k \tau_k(x;0)},
\]
and it follows that
\begin{align}
\label{W11}
\begin{aligned}
	a_n&=\frac{(-1)^n}{n!}\sum_{k=n}^\infty{b_k\left.\D_\W^n x^k\right|_{x=0^{+(n)}}}=\sum_{k=n}^\infty{(-1)^k b_k T(k,n)}
\end{aligned}
\end{align}
where the notation $T(k,n)$ is as in Proposition~\ref{C}.\ \ Next let $f^*$ be the function defined by the Wilson series
\[
	f^*(x):=\sum_{k=0}^\infty{(-1)^k|a_k|\tau_k(x;0)}
\]
and let $\{b_n^*\}_{n\in\mathbb{N}_0}$ be the sequence of real numbers such that $\displaystyle f^*(x)=\sum_{k=0}^\infty{b_k^* x^k}$.\ \ Note that for each pair of positive integers $k\ge n$ we have
\begin{align}
	\nonumber \frac{1}{n!}\left.\frac{d^n}{dx^n}(-1)^k\tau_k(x;0)\right|_{x=0} &= \mbox{coefficient of $x^{k-n}$ in $x(x+1^2)\cdots(x+(k-1)^2)$} \\
	\label{W0} &\le\left(\frac{(k-1)!}{(n-1)!}\right)^2\frac{k!}{n!(k-n)!}.
\end{align}

Since (i) and (ii) follow trivially if $f$ is (non-zero) constant, we assume from now on that $f$ is a non-constant function.\ \ Now we prove (i), and we divide the proof into the following four steps.\ \ The proof of (ii) is essentially step (3).
\begin{enumerate}
\item We first show that $f^*$ is entire and $\sigma_{f^*}\le\sigma$.\ \ Let $\gamma\in(2,\frac{1}{\sigma})$ be arbitrary.\ \ By Lindel\"{o}f-Pringsheim theorem \cite{Boas} we have
\[
	\sigma=\limsup_{k\to\infty}{\frac{k\ln k}{-\ln|b_k|}},
\]
so $|b_k|<k^{-\gamma k}$ for every sufficiently large $k$.\ \ Applying this and Proposition~\ref{C} (ii) to \eqref{W11}, we see that there exist positive constants $K_1$ and $K_2$ such that for every sufficiently large $n\in\mathbb{N}$,
\begin{align*}
	|a_n| &\le \sum_{k=n}^\infty{|b_k|T(k,n)} \\
	&\le K_1\frac{e^{2n}}{n^{2n}}\sum_{k=n}^\infty{\frac{n^{2k}}{k^{\gamma k}}} \le K_1\frac{e^{2n}}{n^{2n}}\sum_{k=n}^\infty{(n^{2-\gamma})^k} \\
	&\le K_2\frac{e^{2n}}{n^{2n}}(n^{2-\gamma})^n = K_2\frac{e^{2n}}{n^{\gamma n}}.
\end{align*}
Applying \eqref{W0} and Stirling's approximation, we see that there exist positive constants $K_3$ and $K_4$ such that for every sufficiently large $n\in\mathbb{N}$,
\begin{align*}
	b_n^* &= \frac{1}{n!}\sum_{k=n}^\infty{|a_k|\left.\frac{d^n}{dx^n}(-1)^k\tau_k(x;0)\right|_{x=0}} \\
	&\le K_2\sum_{k=n}^\infty{\frac{e^{2k}}{k^{\gamma k}}\left(\frac{(k-1)!}{(n-1)!}\right)^2\frac{k!}{n!(k-n)!}} \\
	&= \frac{K_2 n^2}{(n!)^3}\sum_{k=n}^\infty{\frac{e^{2k}((k-1)!)^2}{k^{\gamma k}}\frac{k!}{(k-n)!}} \\
	&\le \frac{K_3 n^2}{(n!)^3}\sum_{k=n}^\infty{\frac{1}{k^{(\gamma-2)k}}k^n} \le \frac{K_4 n^{\frac12}e^{3n}}{n^{\gamma n}}.
\end{align*}
This shows that $f^*$ is an entire function of order
\[
	\sigma_{f^*} \le \frac{1}{\gamma}.
\]
Since $\gamma\in(2,\frac{1}{\sigma})$ was arbitrary, we have $\sigma_{f^*}\le\sigma$.

\bigskip
\item We next show that $\sigma_{\mu_\W(\cdot;f)}\le\sigma_{f^*}$.\ \ For every $r>0$, writing $N:=\nu_\W(r;f)$ we have
\[
	\mu_\W(r;f)=|a_N|r(r+1^2)\cdots(r+(N-1)^2)\le f^*(r)\le M(r;f^*),
\]
so we immediately obtain
\[
	\sigma_{\mu_\W(\cdot;f)}=\limsup_{r\to\infty}{\frac{\ln^+\ln^+\mu_\W(r;f)}{\ln r}}\le \limsup_{r\to\infty}{\frac{\ln^+\ln^+ M(r;f^*)}{\ln r}} =\sigma_{f^*}.
\]

\item Now we show that $\displaystyle\limsup_{r\to\infty}{\frac{\ln^+\nu_\W(r;f)}{\ln r}}\le\sigma_{\mu_\W(\cdot;f)}$.\ \ For every $r>1$ and every $R>r$, writing $N:=\nu_\W(r;f)$, we have
\begin{align*}
	\left[\frac{R+(N-1)^2}{r+(N-1)^2}\right]^N &\le \frac{R}{r}\frac{R+1^2}{r+1^2}\cdots\frac{R+(N-1)^2}{r+(N-1)^2} \\
	&= \frac{|\tau_N(R;0)|}{|\tau_N(r;0)|} \\
	&\le \frac{\mu_\W(R;f)}{\mu_\W(r;f)}.
\end{align*}
By Lemma~\ref{munuprop} (ii) we have $\mu_\W(r;f)\ge 1$ for every sufficiently large $r$, so
\begin{align}
	\label{W8} N\ln\frac{R+(N-1)^2}{r+(N-1)^2} \le \ln^+\mu_\W(R;f),
\end{align}
In particular putting $R=2r+(N-1)^2$ in \eqref{W8}, taking $\ln^+$ and dividing by $\ln r$ on both sides, we arrive at
\[
	\frac{\ln^+ N + \ln\ln 2}{\ln r} \le \frac{\ln^+\ln^+\mu_\W(2r+(N-1)^2;f)}{\ln(2r+(N-1)^2)}\frac{\ln(2r+(N-1)^2)}{\ln r}.
\]

Now for every $\gamma<\frac{1}{\sigma_{\mu_\W(\cdot;f)}}$ and every sufficiently large $r$, we have
\begin{align}
	\label{W9}\frac{\ln^+ N + \ln\ln 2}{\ln r} \le \frac{1}{\gamma}\frac{\ln(2r+(N-1)^2)}{\ln r}.
\end{align}
We claim that $N^2\le r$ for every sufficiently large $r$, so that \eqref{W9} will give
\[
	\frac{\ln^+ N}{\ln r} \le \frac{1}{\gamma}\frac{\ln 3r}{\ln r}-\frac{\ln\ln 2}{\ln r}
\]
for every sufficiently large $r$, which implies the desired inequality on taking limit superior as $r\to\infty$.\ \ To prove this claim, we observe that if on the contrary there exists some sequence $\{r_n\}_{n\in\mathbb{N}}$ of positive real numbers increasing to $\infty$ such that $\nu_\W(r_n;f)^2>r_n$ for every $n\in\mathbb{N}$, then \eqref{W9} gives
\[
	\frac{1}{2}<\frac{\ln \nu_\W(r_n;f)}{\ln r_n}\le \frac{1}{\gamma}\frac{2\ln \nu_\W(r_n;f)+\ln 3}{\ln r_n} - \frac{\ln\ln 2}{\ln r_n},
\]
and since $\sigma<\frac{1}{2}$ enables one to choose $\gamma>2$, this implies that
\[
	\frac{1}{2}\left(1-\frac{2}{\gamma}\right)<\left(1-\frac{2}{\gamma}\right)\frac{\ln \nu_\W(r_n;f)}{\ln r_n}\le \frac{\frac{\ln 3}{\gamma}-\ln\ln 2}{\ln r_n}
\]
for every $n\in\mathbb{N}$, and so as $n\to\infty$ we get $0\ge\frac{1}{2}(1-\frac{2}{\gamma})>0$, a contradiction.

\item Finally we show that $\displaystyle\sigma_{\mu_\W(\cdot;f)}\le\limsup_{r\to\infty}{\frac{\ln^+\nu_\W(r;f)}{\ln r}}$.\ \ Write $N:=\nu_\W(r;f)$ for every $r>0$.\ \ Since $|a_N|\le 1$ for every sufficiently large $r$, we have
\begin{align*}
	\ln\mu_\W(r;f)&= \ln |a_N| + \ln r + \ln(r+1^2) + \cdots + \ln(r+(N-1)^2) \\
	&=\ln |a_N| + N\ln r + \sum_{k=1}^{N-1}{\ln\left(1+\frac{k^2}{r}\right)} \\
	&\le N\ln r + \sum_{k=1}^{N-1}{\ln(1+k^{2-\gamma})} \\
	&= N(\ln r + O(N^{2-\gamma}))
\end{align*}
as $r\to\infty$.\ \ This proves the desired inequality.
\end{enumerate}
\hfill \qed

\bigskip
The following is a new Wilson series analogue of the Lindel\"{o}f-Pringsheim theorem.\ \ It relates the order of the maximal term of a Wilson series to the rate of decay of its coefficients.\ \ In fact one can apply the same technique to obtain a similar result for Newton series under the setting in Ishizaki and Yanagihara's \cite{IY}.
\begin{theorem}
\label{WLP}
Let $\displaystyle f(x):=\sum_{k=0}^\infty{a_k\tau_k(x;0)}$ be an entire function of order $\sigma<\frac12$.\ \ Then
\[
	\sigma_{\mu_\W(\cdot;f)}=\limsup_{n\to\infty}{\frac{n\ln n}{-\ln|a_n|}}.
\]
Moreover if $\sigma<\frac13$, then $\displaystyle\sigma=\limsup_{n\to\infty}{\frac{n\ln n}{-\ln|a_n|}}$.
\end{theorem}
\begin{proof}
We denote $\displaystyle L:=\limsup_{n\to\infty}{\frac{n\ln n}{-\ln|a_n|}}$.
\begin{enumerate}
\item We first show that $L\le\sigma_{\mu_\W(\cdot;f)}$.\ \ By Lemma~\ref{muorder} (i) we have $\sigma_{\mu_\W(\cdot;f)}<\frac{1}{2}<+\infty$, so we let $\alpha\in(\sigma_{\mu_\W(\cdot;f)},+\infty)$ be arbitrary.\ \ Then for every sufficiently large $n$, we have
\[
	|a_n|n^{\frac{1}{\alpha}}(n^{\frac{1}{\alpha}}+1^2)\cdots(n^{\frac{1}{\alpha}}+(n-1)^2) \le \mu_\W(n^{\frac{1}{\alpha}};f) \le e^n.
\]
So
\begin{align*}
	\ln|a_n| &\le n-\ln n^{\frac{1}{\alpha}} - \ln(n^{\frac{1}{\alpha}}+1^2) - \cdots - \ln(n^{\frac{1}{\alpha}}+(n-1)^2) \\
	&\le n - n\ln n^{\frac{1}{\alpha}} \\
	&= n - \frac{n\ln n}{\alpha} \\
	&= -\frac{n\ln n}{\alpha}(1+o(1))
\end{align*}
as $n\to\infty$.\ \ This gives
\[
	\frac{-\ln|a_n|}{n\ln n} \ge \frac{1}{\alpha}(1+o(1))
\]
as $n\to\infty$, and so $L\le\alpha$.\ \ Since $\alpha\in(\sigma_{\mu_\W(\cdot;f)},+\infty)$ was arbitrary, we have $L\le\sigma_{\mu_\W(\cdot;f)}$.
\item Next we show that $\sigma_{\mu_\W(\cdot;f)}\le L$.\ \ By the last paragraph and Lemma~\ref{muorder} (i), we have $L<\frac{1}{2}$, so we let $\beta\in(2,\frac{1}{L})$ be arbitrary.\ \ Then $|a_n|\le n^{-\beta n}$ for every sufficiently large $n$.\ \ Now for each $r>0$, since $\beta>2$, we have $n^\beta-n^2\ge\frac{1}{2}n^\beta\ge r$ for every sufficiently large $n$, so for these $n$ we have
\begin{align*}
	|a_n|r(r+1^2)\cdots(r+(n-1)^2) &\le n^{-\beta n}r(r+1^2)\cdots(r+(n-1)^2) \\
	&\le (r+n^2)^{-n}r(r+1^2)\cdots(r+(n-1)^2) \\
	&\le 1.
\end{align*}
Let $\displaystyle a:=\max_{n\in\mathbb{N}_0}{|a_n|}$.\ \ Then for every sufficiently large $r>0$, we have
\begin{align*}
	\mu_\W(r;f) &= \max\{|a_n|r(r+1^2)\cdots(r+(n-1)^2):n\le(2r)^{\frac{1}{\beta}}\} \\
	&\le ar^{(2r)^{\frac{1}{\beta}}+1},
\end{align*}
and so
\[
	\sigma_{\mu_\W(\cdot;f)} = \limsup_{r\to\infty}{\frac{\ln^+\ln^+ \mu_\W(r;f)}{\ln r}} \le \frac{1}{\beta}.
\]
Since $\beta\in(2,\frac{1}{L})$ was arbitrary, we have $\sigma_{\mu_\W(\cdot;f)}\le L$.
\end{enumerate}
If $\sigma<\frac13$, then $\sigma_{\mu_\W(\cdot;f)}=\sigma$ by Theorem~\ref{muorderequal}, so the final statement follows.
\end{proof}

\section{Proofs of the main results}
\label{sec:Proofs}

In the remainder of this paper, we will focus on entire functions of order smaller than $\frac{1}{3}$ and follow an approach that is similar to \cite{IY}, which deals with Newton series expansions.\ \ We will show that an entire function $f$ behaves locally like a polynomial consisting of the few terms around the maximal term in its Wilson series expansion.\ \ To do this, we write $N:=\nu_\W(r;f)$ and aim to show that those terms $a_n\tau_n(x;0)$ in the Wilson series that are far away from the maximal term $a_N\tau_N(x;0)$ are small, by defining \textit{comparison sequences} $\{\alpha_n\}_n$ and $\{\rho_n\}_n$ and comparing the ratio $\left|\frac{a_n\tau_n(r;0)}{a_N\tau_N(r;0)}\right|$ with $\frac{\alpha_n\rho_N^n}{\alpha_N\rho_N^N}$, whose growth can be controlled.\ \ The construction of the comparison sequences follows from K\"{o}vari \cite{Kovari2}.

\begin{definition}
\label{testseq}
	In the remainder of this paper, we pick a $\delta>0$ and define \textbf{\textit{comparison sequences}} $\{\alpha_n\}_{n\in\mathbb{N}_0}$ and $\{\rho_n\}_{n\in\mathbb{N}_0}$ by
	\[
		\alpha_n:=e^{\int_0^n{\alpha(t)dt}} \hspace{20px}\mbox{and}\hspace{20px} \rho_n:=e^{-\alpha(n)},
	\]
	where $\alpha:[0,\infty)\to\mathbb{R}$ is a $C^1$ function which is linear on $[0,t_0]$ and satisfies that $\alpha'(t)=-\frac{1}{t\ln t(\ln\ln t)^{1+\delta}}$ on $[t_0,+\infty)$, and $t_0\ge e^e$ is a number such that the range of $\alpha$ is contained in $[-\ln 2, 0]$.
\end{definition}

We immediately see that $\{\alpha_n\}_{n\in\mathbb{N}_0}$ is a logarithmically convex sequence.\ \ We also have $\rho_0\in(1,\frac{\alpha_0}{\alpha_1})$ and $\rho_n\in(\frac{\alpha_{n-1}}{\alpha_n},\frac{\alpha_n}{\alpha_{n+1}})$ for every $n\in\mathbb{N}$, so that $\{\rho_n\}_{n\in\mathbb{N}_0}$ is an increasing sequence bounded above by $2$.

\bigskip
We are interested in only those radii $r$ on which $\left|\frac{a_n\tau_n(r;0)}{a_N\tau_N(r;0)}\right|$ can be controlled by $\frac{\alpha_n\rho_N^n}{\alpha_N\rho_N^N}$, so we give them a name.

\begin{definition}
\label{taunormal}
Let $\displaystyle f(x):=\sum_{k=0}^\infty{a_k \tau_k(x;0)}$ be an entire function of order $\sigma<\frac{1}{3}$ and let $\gamma\in(3,\frac{1}{\sigma})$.\ \ A positive real number $r$ is said to be \textbf{\textit{$\tau$-normal}} (for the Wilson series $f$, with respect to $\gamma$ and the comparison sequences $\{\alpha_n\}_{n\in\mathbb{N}_0}$ and $\{\rho_n\}_{n\in\mathbb{N}_0}$) if there exists $N\in\mathbb{N}$ such that for every $n\in\mathbb{N}_0$,
\begin{align*}
	|a_n \tau_n(r;0)| &\le |a_N \tau_N(r;0)|\frac{\alpha_n}{\alpha_N}\rho_N^{n-N} \hspace{60px}&&\mbox{if $n\ge N$, and} \\
	|a_n \tau_n(r;0)| &\le |a_N \tau_N(r;0)|(1+\varepsilon_{n,N})\frac{\alpha_n}{\alpha_N}\rho_N^{n-N} &&\mbox{if $n<N$},
\end{align*}
where $\varepsilon_{n,N}:=\frac{n^2}{N^\gamma}+\cdots+\frac{(N-1)^2}{N^\gamma}<\frac{1}{3N^{\gamma-3}}$.\ \ Positive real numbers that are not $\tau$-normal are said to be \textbf{\textit{$\tau$-exceptional}}.
\end{definition}

The inequality requirements in Definition~\ref{taunormal} are motivated by the following theorem, which asserts that most non-negative numbers are $\tau$-normal.

\begin{theorem}
\label{tauexceptional}
Let $f$ be an entire function of order $\sigma<\frac{1}{3}$ and let $\gamma\in(3,\frac{1}{\sigma})$.\ \ Then the set
\[
	E:=\{r\in[1,\infty): \mbox{r is $\tau$-exceptional for the Wilson series $f$}\}
\]
has finite logarithmic measure.
\end{theorem}
\begin{proof}
We write $\displaystyle f(x)=\sum_{k=0}^\infty{a_k \tau_k(x;0)}$.\ \ Since $\nu_\W(\cdot;f)$ is integer-valued, non-decreasing and right-continuous by Lemma~\ref{munuprop}, we let $\{r_n\}_{n\in\mathbb{N}_0}$ be the monotonic increasing sequence of non-negative numbers such that $r_0:=0$ and $\nu_\W(r;f)=n$ for every $r\in[r_n,r_{n+1})\setminus\{0\}$.\ \ (If $n$ is not in the range of $\nu_\W(\cdot;f)$, then $r_{n+1}=r_n$.)

Now by the choice of $\{r_n\}_{n\in\mathbb{N}_0}$ and the continuity of $\mu_\W(\cdot;f)$ by Lemma~\ref{munuprop}, for every $j\in\mathbb{N}_0$ and $k\in\mathbb{N}$ satisfying $r_j<r_{j+1}=\cdots=r_{j+k}$, we have
\begin{align*}
	&\ \ \ \,|a_{j+k}|r_{j+k}(r_{j+k}+1^2)\cdots(r_{j+k}+(j+k-1)^2) \\
	&\le \mu_\W(r_{j+k};f) = \mu_\W(r_{j+1};f) \\
	&= \lim_{r\to r_{j+1}^-}{\mu_\W(r;f)} = \lim_{r\to r_{j+1}^-}{|a_j|r(r+1^2)\cdots(r+(j-1)^2)} \\
	&= |a_j|r_{j+1}(r_{j+1}+1^2)\cdots(r_{j+1}+(j-1)^2) \\
	&= |a_j|r_{j+k}(r_{j+k}+1^2)\cdots(r_{j+k}+(j-1)^2).
\end{align*}
This gives
\begin{align*}
	\frac{|a_{j+k}|}{|a_j|} &\le \frac{1}{(r_{j+k}+j^2)\cdots(r_{j+k}+(j+k-1)^2)} \\
	&= \frac{1}{(r_{j+1}+j^2)\cdots(r_{j+k}+(j+k-1)^2)}
\end{align*}
whenever $r_j<r_{j+1}=\cdots=r_{j+k}$.\ \ So for every $n\in\mathbb{N}_0$, taking products for the appropriate $j$'s we get
\begin{align}
\label{W1}\frac{|a_n|}{|a_0|} \le \frac{1}{r_1(r_2+1^2)\cdots(r_n+(n-1)^2)}.
\end{align}

Since $\rho_n\in(\frac{\alpha_{n-1}}{\alpha_n},\frac{\alpha_n}{\alpha_{n+1}})$ for every $n\in\mathbb{N}$, we have
\begin{align}
\label{W2}\frac{\alpha_n}{\alpha_0}\ge\frac{1}{\rho_1\rho_2\cdots\rho_n}
\end{align}
and so combining \eqref{W1} and \eqref{W2} we obtain
\[
	\frac{|a_n|}{\alpha_n}\le\frac{|a_0|}{\alpha_0}\frac{\rho_1}{r_1}\frac{\rho_2}{r_2+1^2}\cdots\frac{\rho_n}{r_n+(n-1)^2}.
\]
Lemma~\ref{muorder} (ii) implies that $r_n>n^\gamma$ for every sufficiently large $n$, so there exists $K_0>0$ such that
\begin{align}
	\label{W13}A_n:=\frac{|a_n|}{\alpha_n}\le K_0\frac{|a_0|}{\alpha_0}\frac{2^n}{(n!)^\gamma}
\end{align}
for every sufficiently large $n$.\ \ \eqref{W13} implies that
\[
	\limsup_{n\to\infty}\frac{n\ln n}{-\ln|A_n|} \le \lim_{n\to\infty}\frac{n\ln n}{\gamma\ln(n!)-n\ln 2-\ln(K_0\frac{|a_0|}{\alpha_0})} = \frac{1}{\gamma}
\]
so by Theorem~\ref{WLP}, the function $F$ defined by the Wilson series
\[
	F(x):=\sum_{k=0}^\infty{A_k\tau_k(x;0)}
\]
is an entire function of order at most $\frac{1}{\gamma}$.

Now suppose that $\rho>0$ and that $M=\nu_\W(\rho;F)\ge 1$.\ \ Then noting that $1<\rho_M<2$, for every $n>M$ we have
{\small
\begin{align*}
	\frac{|a_n\tau_n(\rho\rho_M;0)|}{|a_M\tau_M(\rho\rho_M;0)|} &= \frac{\alpha_n A_n|\tau_n(\rho\rho_M;0)|}{\alpha_M A_M|\tau_M(\rho\rho_M;0)|} = \frac{\alpha_n A_n}{\alpha_M A_M}(\rho\rho_M +M^2)\cdots(\rho\rho_M + (n-1)^2) \\
	&\le \frac{\alpha_n A_n}{\alpha_M A_M}(\rho +M^2)\cdots(\rho + (n-1)^2)\rho_M^{n-M} = \frac{\alpha_n A_n|\tau_n(\rho;0)|}{\alpha_M A_M|\tau_M(\rho;0)|}\rho_M^{n-M} \\
	&\le \frac{\alpha_n}{\alpha_M}\rho_M^{n-M} < 1,
\end{align*}}%
while for every $n<M$, since $\displaystyle\sum_{k=n}^{M-1}{\frac{k^2}{M^\gamma}}<\frac{1}{3M^{\gamma-3}}<1$, we have
{\small
\begin{align*}
	\frac{|a_n\tau_n(\rho\rho_M;0)|}{|a_M\tau_M(\rho\rho_M;0)|} &= \frac{\alpha_n A_n|\tau_n(\rho\rho_M;0)|}{\alpha_M A_M|\tau_M(\rho\rho_M;0)|} = \frac{\alpha_n A_n}{\alpha_M A_M}\frac{1}{(\rho\rho_M +n^2)\cdots(\rho\rho_M + (M-1)^2)} \\
	&= \frac{\alpha_n A_n|\tau_n(\rho;0)|}{\alpha_M A_M|\tau_M(\rho;0)|}\rho_M^{n-M}\frac{(\rho\rho_M +\rho_M n^2)\cdots(\rho\rho_M + \rho_M(M-1)^2)}{(\rho\rho_M +n^2)\cdots(\rho\rho_M + (M-1)^2)} \\
	&\le \frac{\alpha_n}{\alpha_M}\rho_M^{n-M}\frac{(1 + \frac{n^2}{\rho})\cdots(1 + \frac{(M-1)^2}{\rho})}{(1 + \frac{n^2}{\rho\rho_M})\cdots(1 + \frac{(M-1)^2}{\rho\rho_M})} \le \frac{\alpha_n}{\alpha_M}\rho_M^{n-M}\prod_{k=n}^{M-1}{\left(1 + \frac{k^2}{2\rho}\right)} \\
	&\le \frac{\alpha_n}{\alpha_M}\rho_M^{n-M}\prod_{k=n}^{M-1}{\left(1 + \frac{k^2}{2M^\gamma}\right)} \le \frac{\alpha_n}{\alpha_M}\rho_M^{n-M}(1+\varepsilon_{n,M}),
\end{align*}}%
where in the third last step we have used the inequality $\frac{1+2a}{1+a}\le 1+a$ which holds for every $a>0$, in the second last step we have used the inequality $\rho>M^\gamma$ which follows from Lemma~\ref{muorder} (ii), and in the last step we have used the inequality $\prod_k{(1+\frac{\lambda_k}{2})}\le 1+\sum_k{\lambda_k}$ which holds for every sequence $\{\lambda_k\}_k$ of non-negative numbers with $\sum_k{\lambda_k}<1$.\ \ We have thus shown that $r$ is $\tau$-normal for $f$ if there exists $\rho>0$ such that $r=\rho\rho_M$ where $M=\nu_\W(\rho;F)$, i.e. if there exists $M\in\mathbb{N}_0$ such that $\nu_\W(\frac{r}{\rho_M};F)=M$.\ \ Therefore if we let $\{R_n\}_{n\in\mathbb{N}}$ be the monotonic increasing sequence such that $\nu_\W(R;F)=n$ for every $R\in[R_n,R_{n+1})$, then
\[
	E\subseteq \bigcup_{k\in\mathbb{N}}{[R_k\rho_{k-1},R_k \rho_k)}.
\]
Now for every $r\in[R_n \rho_n, R_{n+1}\rho_n)$, we have $r=R\rho_n$ for some $R\in[R_n,R_{n+1})$ and so $\nu_\W(r;f)=n$ by the above computations.\ \ So by the definition of $\{r_n\}_{n\in\mathbb{N}_0}$ we have $r_n\le R_n\rho_n$.\ \ Therefore whenever $r\in[r_n,r_{n+1})$, i.e. $\nu_\W(r;f)=n$, we must have $r<R_{n+2}\rho_{n+1}$, and so
\[
	E\cap[1,r] \subseteq E\cap[1,R_{n+2}\rho_{n+1}) \subseteq \bigcup_{k=1}^{n+1}{[R_k\rho_{k-1},R_k \rho_k)},
\]
which implies that
\[
	\logmea(E\cap[1,r]) \le \sum_{k=1}^{n+1}{\int_{R_k\rho_{k-1}}^{R_k\rho_k}{\frac{dt}{t}}} = \ln\frac{\rho_{n+1}}{\rho_0}.
\]
Since $\{r_n\}_{n\in\mathbb{N}_0}$ is unbounded and $\{\rho_n\}_{n\in\mathbb{N}_0}$ is bounded above, it follows that $\logmea E<+\infty$.
\end{proof}

We call the set $E$ in Theorem~\ref{tauexceptional} the \textit{$\tau$-exceptional set} for $f$.\ \ We note that $E$ depends not only on $f$, but also on the choice of $\gamma$ as well as the construction of the comparison sequences $\{\alpha_n\}$ and $\{\rho_n\}$ (which depends on the choice of $\delta$).

\begin{lemma}
\label{WVestimate}
Let $\displaystyle f(x)=\sum_{k=0}^\infty{a_k \tau_k(x;0)}$ be a transcendental entire function of order $\sigma<\frac13$, let $\gamma\in(3,\frac{1}{\sigma})$, and let $E$ be the $\tau$-exceptional set for $f$.\ \ Then for every sufficiently large $r\in(0,\infty)\setminus E$ we have
\[
	\frac{|a_{N+k}\tau_{N+k}(r;0)|}{\mu_\W(r;f)} \le e^{-\frac{1}{2}k^2 b(N+k)}
\]
for every $k\in\mathbb{N}$ and
\[
	\frac{|a_{N-k}\tau_{N-k}(r;0)|}{\mu_\W(r;f)} \le \left(1+\frac{1}{3N^{\gamma-3}}\right)e^{-\frac{1}{2}k^2 b(N)}
\]
for every $k\in\{0,1,\ldots,N-1\}$, where $N=\nu_\W(r;f)$ and $b(N):=\frac{1}{N\ln N (\ln\ln N)^{1+\delta}}$.
\end{lemma}
\begin{proof}
Since $\displaystyle\lim_{r\to\infty}N=+\infty$, we let $r>0$ be sufficiently large so that $N\ge t_0$, where $t_0$ is the number as in Definition~\ref{testseq}.\ \ From the definition of the comparison sequences $\{\alpha_n\}_{n\in\mathbb{N}_0}$ and $\{\rho_n\}_{n\in\mathbb{N}_0}$, we have
\begin{align*}
	\frac{\alpha_{N+k}}{\alpha_N}\rho_N^k &= e^{\int_N^{N+k}{(\alpha(t)-\alpha(N))\,dt}} \le e^{\int_N^{N+k}{(t-N)\alpha'(t)\,dt}} \\
	&\le e^{-\frac{1}{2}k^2\min\{|\alpha'(t)|:t\in[N,N+k]\}} = e^{-\frac{1}{2}k^2 b(N+k)}
\end{align*}
for every $k\in\mathbb{N}$, and
\begin{align*}
	\frac{\alpha_{N-k}}{\alpha_N}\rho_N^{-k} &= e^{-\int_{N-k}^N{(\alpha(t)-\alpha(N))\,dt}} \le e^{-\int_{N-k}^N{(t-N)\alpha'(t)\,dt}} \\
	&\le e^{-\frac{1}{2}k^2\min\{|\alpha'(t)|:t\in[N-k,N]\}} = e^{-\frac{1}{2}k^2 b(N)}
\end{align*}
for every $k\in\{0,1,\ldots,N-1\}$.\ \ So the result follows from Definition~\ref{taunormal}.
\end{proof}

We are now ready to prove Theorem~\ref{muorderequal} as well as the main Theorem~\ref{tail}.

\bigskip
\noindent \textit{Proof of Theorem~\ref{muorderequal}.}\ \ Given an entire function $f$ of order $\sigma<\frac13$, we let $\gamma\in(3,\frac{1}{\sigma})$ and let $E$ be the $\tau$-exceptional set for $f$.\ \ Then for every $\varepsilon>0$, one can deduce that
\begin{align}
\label{Mbound}
	\mu_\W(r;f) \le K(r)M(r;f) \le \mu_\W(r;f) [\ln^+\mu_\W(r;f)]^{\frac12+\varepsilon}
\end{align}
for every sufficiently large $r\in(0,\infty)\setminus E$, where $K(r):=K_{\nu_\W(r;f)}$ as defined in the proof of Lemma~\ref{gg}, so that $K(r)$ decreases to $1$ as $r\to\infty$.\ \ These inequalities \eqref{Mbound} are clear if $f$ is a polynomial.\ \ If $f$ is transcendental, then the first inequality in \eqref{Mbound} follows from Lemma~\ref{muorder} (ii) and Lemma~\ref{gg}, while the second inequality in \eqref{Mbound} follows from Lemma~\ref{WVestimate} and similar arguments as in \cite[pp. 330--334]{Hayman2}.\ \ These inequalities \eqref{Mbound} show that $\sigma_{\mu_\W(\cdot;f)}=\sigma$.
\hfill \qed

\bigskip
\noindent \textit{Proof of Theorem~\ref{tail}.}\ \ The proof is similar to the one of \cite[Theorem 3.3]{IY}.\ \ We take $E$ to be the $\tau$-exceptional set for $f$.\ \ Then we let $\eta\in(0,\frac12]$ be a number to be determined later, and divide the sum into four parts,
\begin{align*}
	&\ \ \ \,\sum_{k:|k-N|\ge \kappa}{k^h|a_k\tau_k(r;0)|} \\
	&= \left(\sum_{k:k\le(1-\eta)N}+\sum_{k:(1-\eta)N<k\le N-\kappa}+\sum_{k:N+\kappa\le k<(1+\eta)N}+\sum_{k:k\ge(1+\eta)N}\right){k^h|a_k\tau_k(r;0)|}.
\end{align*}
\begin{enumerate}[(i)]
	\item For $k\ge(1+\eta)N$ and $r\notin E$, let $p:=k-N$.\ \ Lemma~\ref{WVestimate} gives
		\[
			\frac{k^h|a_k\tau_k(r;0)|}{|a_N \tau_N(r;0)|} \le e^{-\frac{1}{2}p^2 b(N+p)+h\ln(N+p)}.
		\]
		Since $\displaystyle\lim_{r\to\infty}N=+\infty$ by Lemma~\ref{munuprop} and since $p\ge\eta N$, we have
		\begin{align*}
			&\ \ \ \,-\frac{1}{2}p^2 b(N+p)+h\ln(N+p) \\
			&\le -\frac{1}{2}\frac{\eta}{1+\eta}\frac{p}{\ln(N+p)(\ln\ln(N+p))^{1+\delta}} + h\ln p + h\ln\left(\frac{1}{\eta}+1\right) \\
			&\le -\sqrt p
		\end{align*}
		for every sufficiently large $r$.\ \ Therefore
		\begin{align*}
			&\ \ \ \,\frac{1}{|a_N \tau_N(r;0)|}\sum_{k:k\ge(1+\eta)N}{k^h|a_k\tau_k(r;0)|} \le \sum_{p:p\ge\eta N}{e^{-\sqrt p}} \le \int_{\eta N-1}^\infty{e^{-\sqrt t}\,dt} \\
			&= \int_0^\infty{e^{-\sqrt {t+\eta N-1}}\,dt} \le \int_0^\infty{e^{1-\frac{1}{2}\sqrt t-\frac{1}{2}\sqrt{\eta N}}\,dt} = O(e^{-\frac{1}{2}\sqrt{\eta N}})
		\end{align*}
		as $r\to\infty$ and $r\in(0,\infty)\setminus E$.
	\item For $k\le(1-\eta)N$ and $r\notin E$, let $p:=N-k$.\ \ Lemma~\ref{WVestimate} gives
		\[
			\frac{k^h|a_k\tau_k(r;0)|}{|a_N \tau_N(r;0)|} \le \left(1+\frac{1}{3N^{\gamma-3}}\right)e^{-\frac{1}{2}p^2 b(N)+h\ln(N-p)}.
		\]
		Since $\displaystyle\lim_{r\to\infty}N=+\infty$ and since $p\ge\eta N$, we have
		\begin{align*}
			-\frac{1}{2}p^2 b(N)+h\ln(N-p) &\le -\frac{1}{2}\eta\frac{p}{\ln N(\ln\ln N)^{1+\delta}} + h\ln p + h\ln\left(\frac{1}{\eta}-1\right) \\
			&\le -\sqrt p
		\end{align*}
		for every sufficiently large $r$.\ \ Therefore similar to the last paragraph we also have
		\[
			\frac{1}{|a_N \tau_N(r;0)|}\sum_{k:k\le(1-\eta)N}{k^h|a_k\tau_k(r;0)|} = O(e^{-\frac{1}{2}\sqrt{\eta N}})
		\]
		as $r\to\infty$ and $r\in(0,\infty)\setminus E$.
	\item In the remaining case, we let $\varepsilon\in(0,\frac{1}{3N^{\gamma-3}})$ be arbitrary.\ \ Then by the continuity of the function $b$, the number $\eta\in(0,\frac{1}{2}]$ can be chosen small enough so that
		\[
			(1-\eta)^{-|h|}<1+\varepsilon
		\]
		and
		\[
			\frac{b(N+|p|)}{b(N)}>1-\varepsilon \hspace{20px} \mbox{for every $p\in[-\eta N,\eta N]$}.
		\]
		Now for $k\in((1-\eta)N, (1+\eta)N)$ and $r\notin E$, let $p:=k-N$.\ \ Both estimates in Lemma~\ref{WVestimate} give
		\begin{align*}
			\frac{k^h|a_k\tau_k(r;0)|}{|a_N \tau_N(r;0)|} &\le N^h\left(1+\frac{p}{N}\right)^h\left(1+\frac{1}{3N^{\gamma-3}}\right)e^{-\frac{1}{2}p^2b(N+|p|)} \\
			&\le N^h(1+\varepsilon)\left(1+\frac{1}{3N^{\gamma-3}}\right)e^{-\frac{1}{2}p^2(1-\varepsilon)b(N)} \\
			&\le \left(1+\frac{1}{3N^{\gamma-3}}\right)^2 N^h e^{-b^*p^2},
		\end{align*}
		where $b^*:=\frac{1}{2}(1-\varepsilon)b(N)$.
\end{enumerate}
Combining the above three paragraphs, we see that for every $\varepsilon\in(0,\frac{1}{3N^{\gamma-3}})$ and every $\kappa\in\mathbb{N}$, there exists $\eta\in(0,\frac{1}{2}]$ such that
\[
	\sum_{k:|k-N|\ge \kappa}{k^h|a_k\tau_k(r;0)|} \le 2\left(1+\frac{1}{3N^{\gamma-3}}\right)^2 N^h \mu_\W(r;f)\left[\sum_{p=\nu}^\infty{e^{-b^*p^2}}+O(e^{-\frac{1}{2}\sqrt{\eta N}})\right]
\]
as $r\to\infty$ and $r\in(0,\infty)\setminus E$.\ \ Note that
\[
	\sum_{p=\kappa}^\infty{e^{-b^*p^2}} \le \int_{\kappa-1}^\infty{e^{-b^*t^2}\,dt} = \frac{1}{\sqrt{b^*}}\left(\frac{e^{-y_0^2}}{2y_0}-\int_{y_0}^\infty{\frac{e^{-y^2}}{2y^2}\,dy}\right)
\]
where $y_0:=(\kappa-1)\sqrt{b^*}$.\ \ So given any $\beta>0$, if we take $\kappa=\left[\sqrt{\frac{\beta}{b(N)}\ln\frac{1}{b(N)}}\right]$, then for every $\omega\in(0,\beta)$, the number $\varepsilon$ can be chosen so small that
\[
	\sum_{p=\kappa}^\infty{e^{-b^*p^2}} = O\left(\frac{e^{-y_0^2}}{y_0\sqrt{b^*}}\right) = O\left(\frac{e^{\frac{1}{2}(1-\varepsilon)\beta\ln b(N)}}{\sqrt{b(N)\ln\frac{1}{b(N)}}}\right) = o(b(N)^{\frac{\omega-1}{2}})
\]
as $r\to\infty$ and $r\in(0,\infty)\setminus E$.
\hfill \qed

\begin{lemma}
\label{polyn}
Let $r>\frac{1}{4}$ and $p$ be a polynomial of degree $d$.\ \ Then for every $R\ge r$ and every $x\in\overline{D(0;R)}$, we have
\begin{align}
	\label{W14}|(\A_\W p)(x)| \le \frac{(\sqrt{R}+\frac{1}{2})^{2d}}{r^d}M(r;p)
\end{align}
and
\begin{align}
	\label{W15}|(\D_\W p)(x)| \le \frac{2ed(\sqrt{R}+\frac{1}{2})^{2(d-1)}}{r^d}M(r;p).
\end{align}
\end{lemma}
\begin{proof}
Applying maximum principle to $\frac{p(x)}{x^d}$ on $\hat{\mathbb{C}}\setminus D(0;r)$, we have
\begin{align}
\label{Maxprin}
	|p(z)| \le \frac{|z|^d}{r^d}M(r;p)
\end{align}
whenever $|z|>r$.\ \ Now let $R\ge r$.
\begin{enumerate}[(i)]
	\item Applying maximum principle to $p$ on $D(0;(\sqrt{R}+\frac{1}{2})^2)$, \eqref{Maxprin} gives
	\[
		|p(x)|\le M\left(\left(\sqrt{R}+\frac{1}{2}\right)^2;p\right)\le\frac{(\sqrt{R}+\frac{1}{2})^{2d}}{r^d}M(r;p)
	\]
	for every $x\in\overline{D(0;(\sqrt{R}+\frac{1}{2})^2)}$, and so \eqref{W14} follows for every $x\in\partial D(0;R)$.
	\item Applying Cauchy's inequality to \eqref{Maxprin} we have
	\[
		|p'(x)|\le \frac{ed(\sqrt{R}+\frac{1}{2})^{2(d-1)}}{r^d}M(r;p)
	\]
	for every $x\in\overline{D(0;(\sqrt{R}+\frac{1}{2})^2)}$ \cite[Lemma 7, p. 337]{Hayman2}.\ \ Hence for every $x\in\partial D(0;R)$,
	\begin{align*}
		|(\D_\W p)(x)| &= \left|\frac{p(x^+)-p(x^-)}{2i\sqrt{x}}\right| \\
		&\le \int_{-\frac{1}{2}}^{\frac{1}{2}}{\left|p'((\sqrt{x}+it)^2)\frac{2i(\sqrt{x}+it)}{2i\sqrt {x}}\right|\,dt} \\
		&\le \frac{2ed(\sqrt{R}+\frac{1}{2})^{2(d-1)}}{r^d}M(r;p)
	\end{align*}
	which is \eqref{W15}.
\end{enumerate}
Since $\D_\W p$ and $\A_\W p$ are polynomials, both \eqref{W14} and \eqref{W15} still hold for every $x\in\overline{D(0;R)}$ by the maximum principle.
\end{proof}

With Theorem~\ref{tail} and Lemma~\ref{polyn}, we can now prove Theorem~\ref{WVmain}, which is an estimate on the behavior of successive Wilson differences of a transcendental entire function of order smaller than $\frac13$.

\noindent \textit{Proof of Theorem~\ref{WVmain}.}\ \ At each $x\in\mathbb{C}$, we let $b(N):=\frac{1}{N\ln N (\ln\ln N)^{1+\delta}}$ and $\kappa:=\left[\sqrt{\frac{10}{b(N)}\ln\frac{1}{b(N)}}\right]$, and let
\[
	\phi(x):=\sum_{k:|k-N|>\kappa}{a_k\tau_k(x;0)} \hspace{20px} \mbox{and} \hspace{20px} p(x):=\sum_{k=N-\kappa}^{N+\kappa}{a_k\frac{\tau_k(x;0)}{\tau_{N-\kappa}(x;0)}}.
\]
Then locally $p$ is a polynomial of degree at most $2\kappa$ and
\[
	f(x)=\phi(x) + \tau_{N-\kappa}(x;0)p(x).
\]
We take $E$ to be the $\tau$-exceptional set for $f$.
\begin{enumerate}[(i)]
\item Applying Theorem~\ref{tail} with $h=n$, $\beta=10$ and $\omega=9$, we have
\begin{align*}
	r^n|(\D_\W^n\phi)(x)| &= r^n\left|\sum_{k:|k-N|>\kappa}{a_k(-1)^n k(k-1)\cdots(k-n+1)\tau_{k-n}(x;0^{+(n)})}\right| \\
	&\le \sum_{k:|k-N|\ge\kappa}{k^n|a_k\tau_k(r;0)|\frac{r^n}{r(r+1^2)\cdots(r+(n-1)^2)}} \\
	&\le \sum_{k:|k-N|\ge\kappa}{k^n|a_k\tau_k(r;0)|} \\
	&= o(\mu_\W(r;f)N^n b(N)^4) = o(\mu_\W(r;f)N^{n-4})
\end{align*}
as $r\to\infty$ and $r\in(0,\infty)\setminus E$.\ \ In particular, we have
\begin{align}
	\label{W16}|\phi(x)| = o(\mu_\W(r;f)N^{-4}) = o\left(\frac{\kappa}{N}\right)M(r;f)
\end{align}
as $r\to\infty$ and $r\in(0,\infty)\setminus E$.
\item On the other hand, since $\sum_{k=1}^{N-\kappa-1}\frac{k^2}{r}<\frac{(N-\kappa)(N-\kappa-1)(2N-2\kappa-1)}{6N^\gamma}<1$, we have
\begin{align*}
	|\tau_{N-\kappa}(x;0)| &\ge r(r-1^2)\cdots(r-(N-\kappa-1)^2) \\
	&= r^{N-\kappa}\left(1-\frac{1^2}{r}\right)\cdots\left(1-\frac{(N-\kappa-1)^2}{r}\right) \\
	&\ge r^{N-\kappa}\left(1-2\sum_{k=1}^{N-\kappa-1}\frac{k^2}{r}\right) \\
	&\ge r^{N-\kappa}\left(1-\frac{(N-\kappa)(N-\kappa-1)(2N-2\kappa-1)}{3N^\gamma}\right) \\
	&= r^{N-\kappa}(1-\varepsilon)
\end{align*}
where $\varepsilon\to 0$ as $r\to\infty$.\ \ This together with \eqref{W16} gives
\[
	|p(x)|=\frac{1}{|\tau_{N-\kappa}(x;0)|}|f(x)-\phi(x)| \le r^{\kappa-N}(1+\varepsilon')M(r;f),
\]
where $\varepsilon'\to 0$ as $r\to\infty$ and $r\in(0,\infty)\setminus E$.\ \ Setting $M_0:=r^{\kappa-N}(1+\varepsilon')M(r;f)$ and applying \eqref{W15} in Lemma~\ref{polyn}, we have
\[
	|(\D_\W p)(x)| \le \frac{4e\kappa(\sqrt{r}+\frac12)^{4\kappa-2}}{r^{2\kappa}}M_0 = O\left(\frac{\kappa}{r}\right)M_0
\]
as $r\to\infty$ and $r\in(0,\infty)\setminus E$, and inductively for each $j\in\mathbb{N}_0$ we have
\[
	|(\D_\W^j p)(x)| = O\left(\frac{\kappa^j}{r^j}\right)M_0
\]
as $r\to\infty$ and $r\in(0,\infty)\setminus E$.\ \ We also note that for each $j\in\{0,1,\ldots,n\}$,
\begin{align*}
	\D_\W^{n-j}\tau_{N-\kappa}(x;0) &= (-1)^{n-j}\frac{(N-\kappa)!}{(N-\kappa-n+j)!}\tau_{N-\kappa-n+j}(x;0^{+(n-j)}) \\
	&= (-1)^n\frac{(N-\kappa)!}{(N-\kappa-n)!}\tau_{N-\kappa-n}(x;0^{+(n)}) \\
	&\ \ \ \,\cdot\frac{(N-\kappa-n)!}{(N-\kappa-n+j)!}(-1)^j\frac{\tau_{N-\kappa-n+j}(x;0^{+(n-j)})}{\tau_{N-\kappa-n}(x;0^{+(n)})} \\
	&= (-1)^n\frac{(N-\kappa)!}{(N-\kappa-n)!}\tau_{N-\kappa-n}(x;0^{+(n)})O\left(\frac{(r+N^2)^j}{N^j}\right) \\
	&= (-1)^n\frac{(N-\kappa)!}{(N-\kappa-n)!}\tau_{N-\kappa-n}(x;0^{+(n)})O\left(\frac{r^j}{N^j}\right)
\end{align*}
as $r\to\infty$, where the last step followed from Lemma~\ref{muorder} (ii).\ \ These together with Theorem~\ref{WLeibniz}, \eqref{W14} in Lemma~\ref{polyn} and \eqref{W16} yield
{\footnotesize
\begin{align*}
	&\ \ \ \,\D_\W^n(\tau_{N-\kappa}(x;0)p(x)) \\
	&= \sum_{k=0}^n{C(n,k)\sum_{j=0}^{n-k}{\binom{n-k}{j}\A_\W^{n-k-j}\D_\W^{j+k}p(x)\A_\W^j\D_\W^{n-j}\tau_{N-\kappa}(x;0)}} \\
	&= \A_\W^n p(x)\D_\W^n\tau_{N-\kappa}(x;0) + \sum_{j=1}^n{\binom{n}{j}\A_\W^{n-j}\D_\W^j p(x)\A_\W^j\D_\W^{n-j}\tau_{N-\kappa}(x;0)} \\
	&\ \ \ \,+ \sum_{k=1}^n{C(n,k)\sum_{j=0}^{n-k}{\binom{n-k}{j}\A_\W^{n-k-j}\D_\W^{j+k}p(x)\A_\W^j\D_\W^{n-j}\tau_{N-\kappa}(x;0)}} \\
	&= (-1)^n\frac{(N-\kappa)!}{(N-\kappa-n)!}\tau_{N-\kappa-n}(x;0^{+(n)})\left(\A_\W^n p(x)+O\left(\frac{\kappa}{N}\right)M_0\right) \\
	&= (-1)^n\frac{(N-\kappa)!}{(N-\kappa-n)!}\tau_{N-\kappa-n}(x;0^{+(n)})\left(p(x)+O\left(\frac{\kappa}{N}\right)M_0\right) \\
	&= (-1)^n\frac{(N-\kappa)!}{(N-\kappa-n)!}\frac{\tau_{N-\kappa-n}(x;0^{+(n)})}{\tau_{N-\kappa}(x;0)}\left(f(x)-\phi(x)+\tau_{N-\kappa}(x;0)O\left(\frac{\kappa}{N}\right)M_0\right) \\
	&= (-1)^n\frac{(N-\kappa)!}{(N-\kappa-n)!}\frac{\tau_{N-\kappa-n}(x;0^{+(n)})}{\tau_{N-\kappa}(x;0)}\left(f(x)+O\left(\frac{\kappa}{N}\right)M(r;f)\right)
\end{align*}}%
as $r\to\infty$ and $r\in(0,\infty)\setminus E$.
\end{enumerate}
The above paragraphs imply that
\begin{align*}
	\left(\frac{x}{N}\right)^n(\D_\W^n f)(x) &= \left(\frac{x}{N}\right)^n(\D_\W^n\phi)(x) + \left(\frac{x}{N}\right)^n\D_\W^n(\tau_{N-\kappa}(x;0)p(x)) \\
	&= f(x)+O\left(\frac{\kappa}{N}\right)M(r;f)
\end{align*}
as $r\to\infty$ and $r\in(0,\infty)\setminus E$.\ \ In this final conclusion we may replace $\kappa$ by $\kappa=\left[\sqrt{N(\ln N)^2(\ln\ln N)^{1+\delta}}\right]$. \hfill \qed

\section{Applications}
\label{sec:Applications}
Our Wilson version of the Wiman-Valiron theory can be applied when studying difference equations involving the Wilson operator.\ \ One can refer to Z. Chen's book \cite{Chen} for a comprehensive study on basics of complex difference equations.\ \ An \textit{(ordinary) Wilson difference equation} is an equation involving an unknown complex function and its Wilson differences, i.e. an equation of the form
\[
	F(x, y, \D_\W y, \D_\W^2 y, \D_\W^3 y, \ldots) = 0,
\]
in which $y$ is an unknown function of the complex variable $x$.\ \ It is said to be \textit{linear} if $F$ is linear in $y$ and its Wilson differences, i.e. if it is of the form
\[
	a_n\D_\W^n y + \cdots + a_1\D_\W y + a_0 y = 0,
\]
where $a_0,\ldots,a_n$ are given functions of $x$.

\begin{example}
\textup{(Eigenvectors of $\D_\W$)}
It can be readily verified that the simplest linear first-order Wilson difference equation
\[
	\D_\W y = y
\]
has two linearly independent entire solutions given by
\begin{align}
\label{W19}
\begin{aligned}
	f_1(x) &= I_{2i\sqrt x}(2) + I_{-2i\sqrt x}(2) \hspace{20px} \mbox{and} \\
	f_2(x) &= K_{2i\sqrt x}(-2) + K_{-2i\sqrt x}(-2),
\end{aligned}
\end{align}
where $I_\alpha$ and $K_\alpha$ are respectively the modified Bessel function of the first and the second kind with order $\alpha\in\mathbb{C}$, which are defined by
\[
	I_\alpha(z):=\sum_{k=0}^\infty{\frac{1}{k!\,\Gamma(k+\alpha+1)}\left(\frac{z}{2}\right)^{2k+\alpha}}
\]
and
\[
	K_\alpha(z):=\frac{\pi}{2\sin\alpha\pi}\left[I_{-\alpha}(z)-I_{\alpha}(z)\right].
\]
$f_1$ and $f_2$ in \eqref{W19} are eigenfunctions of $\D_\W$ corresponding to the eigenvalue $1$, so they can be regarded as \textit{Wilson analogues of the exponential function}.\ \ In fact, for every $\lambda\in\mathbb{C}\setminus\{0\}$, the functions
\begin{align*}
	f_1(x) &= I_{2i\sqrt x}\left(\frac{2}{\lambda}\right) + I_{-2i\sqrt x}\left(\frac{2}{\lambda}\right) \hspace{20px} \mbox{and} \\
	f_2(x) &= K_{2i\sqrt x}\left(-\frac{2}{\lambda}\right) + K_{-2i\sqrt x}\left(-\frac{2}{\lambda}\right)
\end{align*}
are linearly independent entire solutions to the Wilson difference equation
\[
	\D_\W y = \lambda y.
\]
\end{example}

Theorem~\ref{WVmain} can be used to obtain the following result about linear Wilson difference equations.\ \ The classical analogue of this result about linear differential equations can be found in \cite[\S 4.5]{Valiron3}.
\begin{theorem}
\label{WVDE}
	Let $f$ be a transcendental entire solution of order $\sigma<\frac{1}{3}$ to the Wilson difference equation
	\begin{align}
		\label{WW} a_n\D_\W^n y + \cdots + a_1\D_\W y + a_0 y = 0,
	\end{align}
	where $a_0,\ldots,a_n$ are polynomials and $a_n\not\equiv 0$.\ \ Then the following statements hold:
\begin{enumerate}[(i)]
	\item $\sigma$ is the slope of some edge of the \textbf{Newton polygon} for the Wilson difference equation \eqref{WW}, i.e. the convex hull of
	\[
		\bigcup_{k=0}^n\left\{(x,y)\in\mathbb{R}^2: \mbox{$x\ge k$ and $y\le(\deg a_{n-k})-(n-k)$}\right\}.
	\]
	In particular, we have $\sigma\in(0,\frac13)\cap\mathbb{Q}$.
	\item There exists $L>0$ such that
	\[
		\ln M(r;f) = Lr^{\sigma}(1+o(1))
	\]
	as $r\to\infty$, i.e. $f$ is of finite type.
\end{enumerate}
\end{theorem}

\begin{proof}
Given a solution $f$ of \eqref{WW}, we let
\[
	S:=\{x\in\mathbb{C}:|f(x)|=M(|x|;f)\}.
\]
Then $S$ has non-empty intersection with $\partial D(0;r)$ for every $r>0$.\ \ Substituting $f$ into \eqref{WW} and applying Theorem~\ref{WVmain} to $f$, we have
\[
	\left(a_n\frac{N^n}{x^n} + \cdots + a_1\frac{N}{x} + a_0\right)f(x)(1+o(1)) = 0
\]
uniformly on $S$ as $r=|x|\to\infty$ and $r\in[0,+\infty)\setminus E$, where $N=\nu_\W(r;f)$ and $E$ is the $\tau$-exceptional set for $f$.\ \ So denoting $c_k$ as the leading coefficient of the polynomial $a_k$ for each $k$, we have
\[
	\sum_{k=0}^n{c_k N^k x^{(\deg a_k) - k}{(1+o(1))}}=0
\]
uniformly on $S$ as $r\to\infty$ and $r\in(0,+\infty)\setminus E$.\ \ This implies that
\begin{align}
	\label{W40}N=Lr^\chi(1+o(1))
\end{align}
as $r\to\infty$ and $r\in(0,\infty)\setminus E$ for some $L>0$ and some positive rational number $\chi$ which is the slope of some edge of the Newton polygon for \eqref{WW}.\ \ We have
\[
	\chi=\limsup_{r\to\infty}{\frac{\ln^+\nu_\W(r;f)}{\ln r}}=\sigma_{\mu_\W(r;f)}=\sigma
\]
by Lemma~\ref{muorder} (i) and by \eqref{Mbound}, so statement (i) follows.

\bigskip
Next, since
\begin{align*}
	\ln\mu_\W(r;f) &= \ln|a_N|+\ln r +\ln(r+1^2) +\cdots + \ln(r+(N-1)^2) \\
	&= \ln|a_N| + N\ln r + \sum_{k=1}^{N-1}{\ln\left(1+\frac{k^2}{r}\right)},
\end{align*}
if we let $\{r_j\}_{j\in\mathbb{N}}$ be the monotonic increasing sequence of positive real numbers such that $\nu_\W(r;f)=j$ for all $r\in[r_j,r_{j+1})$, then
\[
	\frac{d}{dr}\ln\mu_\W(r;f) = \frac{j}{r} + \frac{d}{dr}\sum_{k=1}^{j-1}{\ln\left(1+\frac{k^2}{r}\right)}
\]
for all $r\in(r_j,r_{j+1})$, and so
\[
	\ln\mu_\W(r;f)-\ln\mu_\W(r_j;f)=\int_{r_j}^r{\frac{j}{t}\,dt} + \sum_{k=1}^{j-1}{\ln\left(1+\frac{k^2}{r}\right)} - \sum_{k=1}^{j-1}{\ln\left(1+\frac{k^2}{r_j}\right)}
\]
for all $r\in[r_j,r_{j+1}]$.\ \ Now for all $r\in[r_{j+1},r_{j+2}]$ we have
\begin{align}
\label{W18}
\begin{aligned}
	&\ \ \ \,\ln\mu_\W(r;f)\\
	&=\ln\mu_\W(r_{j+1};f)+\int_{r_{j+1}}^r{\frac{j+1}{t}\,dt} + \sum_{k=1}^{j}{\ln\left(1+\frac{k^2}{r}\right)} - \sum_{k=1}^{j}{\ln\left(1+\frac{k^2}{r_{j+1}}\right)} \\
	&= \left[\ln\mu_\W(r_j;f)+\int_{r_j}^{r_{j+1}}{\frac{j}{t}\,dt} + \sum_{k=1}^{j-1}{\ln\left(1+\frac{k^2}{r_{j+1}}\right)} - \sum_{k=1}^{j-1}{\ln\left(1+\frac{k^2}{r_j}\right)}\right] \\
	&\ \ \ \,+\int_{r_{j+1}}^r{\frac{j+1}{t}\,dt} + \sum_{k=1}^{j}{\ln\left(1+\frac{k^2}{r}\right)} - \sum_{k=1}^{j}{\ln\left(1+\frac{k^2}{r_{j+1}}\right)} \\
	&= \ln\mu_\W(r_j;f)+\int_{r_j}^{r}{\frac{\nu_\W(t;f)}{t}\,dt} + \sum_{k=1}^{j}{\ln\left(1+\frac{k^2}{r}\right)} - \sum_{k=1}^{j-1}{\ln\left(1+\frac{k^2}{r_j}\right)} \\
	&\ \ \ \,- \ln\left(1+\frac{j^2}{r_{j+1}}\right) \\
	&= \ln\mu_\W(r_{j-1};f)+\int_{r_{j-1}}^{r}{\frac{\nu_\W(t;f)}{t}\,dt} + \sum_{k=1}^{j}{\ln\left(1+\frac{k^2}{r}\right)} - \sum_{k=1}^{j-2}{\ln\left(1+\frac{k^2}{r_{j-1}}\right)} \\
	&\ \ \ \,- \ln\left(1+\frac{(j-1)^2}{r_j}\right) - \ln\left(1+\frac{j^2}{r_{j+1}}\right) \\
	&= \cdots \\
	&= \ln\mu_\W(r_1;f)+\int_{r_1}^{r}{\frac{\nu_\W(t;f)}{t}\,dt} + \sum_{k=1}^{j}{\ln\left(1+\frac{k^2}{r}\right)} - \sum_{k=1}^{j}{\ln\left(1+\frac{k^2}{r_{k+1}}\right)}.
\end{aligned}
\end{align}
By choosing any $\gamma\in(3,\frac{1}{\sigma})$, we see that for all $r\in[r_{j+1},r_{j+2}]$, the first sum on the right-hand side of \eqref{W18} satisfies
\[
	\sum_{k=1}^{j}{\ln\left(1+\frac{k^2}{r}\right)} \le \sum_{k=1}^{j}{\ln\left(1+\frac{k^2}{(j+1)^\gamma}\right)} \le \sum_{k=1}^{j}{\frac{k^2}{(j+1)^\gamma}} \le \frac{1}{3(j+1)^{\gamma-3}},
\]
and the second sum on the right-hand side of \eqref{W18} satisfies
\[
	\sum_{k=1}^{j}{\ln\left(1+\frac{k^2}{r_{k+1}}\right)} \le \sum_{k=1}^{j}{\ln\left(1+\frac{k^2}{(k+1)^\gamma}\right)} \le \sum_{k=1}^{j}{\frac{1}{(k+1)^{\gamma-2}}}\le \ln (j+1).
\]
So we have
\[
	\sum_{k=1}^{j}{\ln\left(1+\frac{k^2}{r}\right)}=O(N^{3-\gamma}) \hspace{15px}\mbox{and}\hspace{15px} \sum_{k=1}^{j}{\ln\left(1+\frac{k^2}{r_{k+1}}\right)}=o(r^\chi)
\]
as $r\to\infty$ and $r\in(0,\infty)\setminus E$ by \eqref{W40}.
Therefore as $r\to\infty$ and $r\in(0,\infty)\setminus E$, \eqref{W18} becomes
\begin{align*}
	\ln\mu_\W(r;f) &= \ln\mu_\W(r_1;f)+\int_{r_1}^{r}{\frac{\nu_\W(t;f)}{t}\,dt} + O(N^{3-\gamma}) + o(r^\chi) \\
	&= \frac{L}{\chi}r^\chi(1+o(1)).
\end{align*}

Applying \eqref{Mbound} to this asymptotic, we also have
\[
	\ln M(r;f) = \frac{L}{\chi}r^\chi(1+o(1))
\]
as $r\to\infty$ and $r\in(0,\infty)\setminus E$.\ \ Since $\logmea E<+\infty$ by Theorem~\ref{tauexceptional}, one can show by the same arguments as in \cite[pp. 259--261]{HN} that the same asymptotic holds as $r\to\infty$ without any exceptional set.\ \ This finishes the proof of statement (ii).
\end{proof}

\begin{remark}
The conclusions of Theorem~\ref{WVDE} do not hold in general for entire solutions of order at least $\frac12$.\ \ Consider the entire function
\[
	f(x)=\frac{1}{\Gamma(2i\sqrt{x})}+\frac{1}{\Gamma(-2i\sqrt{x})}.
\]
It can be easily verified that
\begin{align}
\label{Eg}
\begin{aligned}
	\D_\W f(x) &= f(x)\left(1+\frac{1}{4x}\right) - g(x) \hspace{20px} \mbox{and} \\
	\D_\W^2 f(x) &= f(x)\left[1+\frac{2}{4x+1}+\frac{4x-1}{4x(4x+1)^2}\right] -g(x)\left[2+\frac{2}{(4x+1)^2}\right],
\end{aligned}
\end{align}
where $g$ is the entire function
\[
	g(x)=\frac{1}{2i\sqrt{x}}\left[\frac{1}{\Gamma(2i\sqrt{x})}-\frac{1}{\Gamma(-2i\sqrt{x})}\right].
\]
Eliminating $g$ from \eqref{Eg}, we find that $f$ is an entire solution to the linear Wilson difference equation
\begin{align}
\label{Egeqn}
	a_2\D_\W^2 y +a_1\D_\W y + a_0y=0,
\end{align}
where
\[
\begin{cases}
	a_0(x)=64x^3+32x^2+16x+5 \\
	a_1(x)=-16x(8x^2+4x+1) \\
	a_2(x)=4x(4x+1)^2
\end{cases}.
\]
This solution $f$ is of order $\frac12$ but of infinite type, and none of the edges of the Newton polygon of \eqref{Egeqn} has slope $\frac{1}{2}$.\ \ So neither of the conclusions of Theorem~\ref{WVDE} hold for $f$.
\end{remark}

\bigskip
In \cite{Chiang-Feng4}, Chiang and Feng have obtained a result for difference equations which works for entire solutions of order smaller than $1$.\ \ The result was established via a direct comparison between $\frac{\Delta f}{f}$ and $\frac{f'}{f}$, which was done without using Newton series at all.\ \ This estimate is more general compared with Ishizaki and Yanagihara's result in \cite{IY}, which only works for entire solutions of order smaller than $\frac12$.\ \ Although one can potentially obtain a better result than Theorem~\ref{WVDE} by following Chiang and Feng's approach, we follow Ishizaki and Yanagihara's approach in this paper because the Wiman-Valiron theory for Wilson series established here is of function theoretic importance.

\section{Discussion}

In this paper, we have investigated an interpolation series expansion of entire functions with respect to a polynomial basis related to the Wilson divided-difference operator.\ \ A convergent series expansion of this type exists for any entire function of order smaller than $\frac{1}{2}$.\ \ Moreover, as Ishizaki and Yanagihara have done for the Newton series expansion \cite{IY}, we have developed in this paper a Wiman-Valiron theory for this type of interpolation series expansions.\ \ A key estimate for those terms which are far away from the maximal term has been established for entire functions of order smaller than $\frac{1}{3}$, and this estimate shows that the local behavior of these functions is mainly contributed by those terms which are near the maximal term.\ \ This key estimate also gives rise to a growth relation between an entire function $f$ and its $n$th Wilson difference $\D_\W^n f$, which can be applied to study difference equations involving the Wilson operator.\ \ Along the way of proving the estimate, we have also got various properties of the Wilson maximal term and central index, and have obtained a Wilson series version of the Lindel\"{o}f-Pringsheim theorem which compares the coefficients of the series with the growth of the maximal term.\ \ Combined with the Nevanlinna theory for the Wilson operator established in \cite{Cheng-Chiang}, we have got better understanding in the function theory behind the Wilson operator.\ \ There are corresponding versions of residue calculus for the Wilson operator as well as the Askey-Wilson divided-difference operator acting on meromorphic functions, and these may provide natural ways to better understand the corresponding special functions.\ \ These issues will be discussed in subsequent papers.

\section*{Acknowledgement}

The author would like to thank his PhD thesis advisor Edmund Chiang for his valuable comments and suggestions in this research.\ \ The author would also like to thank the reviewers for their valuable suggestions that have improved this paper.


\bibliographystyle{amsplain}

\begin{thebibliography}{10}
\bibitem{Andrews-Askey} G. E. Andrews and R. Askey, \textit{Classical orthogonal polynomials}, Orthogonal polynomials and applications (Bar-le-Duc, 1984), 36--62, Lecture Notes in Mathematics \textbf{1171}, Springer, Berlin (1985).
\bibitem{Askey-Wilson} R. Askey and J. A. Wilson, \textit{Some basic hypergeometric orthogonal polynomials that generalize Jacobi polynomials}, Memoirs of the American Mathematical Society \textbf{54} (1985), no. 319, iv+55. MR 783216 (87a:05023).
\bibitem{Boas} R. P. Boas, \textit{Entire functions}, Academic Press, New York (1954).
\bibitem{BSSV} P. L. Butzer, K. Schmidt, E. L. Stark and L. Vogt, \textit{Central factorial numbers; their main properties and some applications}, Numerical Functional Analysis and Optimization \textbf{10} (1989), 419--488.
\bibitem{CR} L. Carlitz and J. Riordan, \textit{The divided central differences of zero}, Canadian Journal of Mathematics \textbf{15} (1963), 94--100.
\bibitem{Charalambides} Ch. A. Charalambides, \textit{Central factorial numbers and related expansions}, The Fibonacci Quarterly \textbf{19.5} (1981), 451--456.
\bibitem{Chen} Z. Chen, \textit{Complex differences and difference equations}, Mathematics Monograph Series \textbf{29}, Science Press, Beijing (2014).
\bibitem{Cheng} K. H. Cheng, \textit{Function theoretic results about the Wilson operator}, PhD Thesis, The Hong Kong University of Science and Technology (2017).
\bibitem{Cheng-Chiang} K. H. Cheng and Y. M. Chiang, \textit{Nevanlinna theory of the Wilson divided-difference operator}, Annales Academi\ae\ Scientiarum Fennic\ae\ Mathematica \textbf{42} (2017), 175--209.
\bibitem{Chiang-Feng4} Y. M. Chiang and S. J. Feng, \textit{On the growth of logarithmic difference of meromorphic functions and a Wiman-Valiron estimate}, Constructive Approximations \textbf{44} (2016), 313--326.
\bibitem{Clunie1} J. Clunie, \textit{The determination of an integral function of finite order by its Taylor series}, Journal of the London Mathematical Society \textbf{28} (1953), 58--66.
\bibitem{Clunie2} J. Clunie, \textit{On the determination of an integral function from its Taylor series}, Journal of the London Mathematical Society \textbf{30} (1955), 32--42.
\bibitem{Cooper} S. Cooper, \textit{The Askey-Wilson operator and the $ _6\phi_5$ summation formula}, South East Asian Journal of Mathematics and Mathematical Sciences \textbf{1} (2002), no. 1, 71--82.
\bibitem{Fenton1} P. C. Fenton, \textit{Some results of Wiman-Valiron type for integral functions of finite lower order}, Annals of Mathematics \textbf{103} (1976), 237--252.
\bibitem{Fenton} P. C. Fenton, \textit{A glance at Wiman-Valiron theory}, Contemporary Mathematics \textbf{382} (2005), 131--139.
\bibitem{GLN} W. Gawronski, L. L. Littlejohn and T. Neuschel, \textit{Asymptotics of Chebyshev-Stirling and Stirling numbers of the second kind}, Studies in Applied Mathematics, \textbf{133} (2014), 1--17.
\bibitem{Gelfond} A. O. Gelfond, \textit{Calculus of finite differences} (in Russian), State Publishing Office for Physical and Mathematical Literature, Moscow (1971).
\bibitem{Grosswald} E. Grosswald, \textit{Bessel polynomials}, Lecture Notes in Mathematics \textbf{698}, Springer-Verlag, Berlin, Heidelberg (1978).
\bibitem{Hayman2} W. K. Hayman, \textit{The local growth of power series: A survey of the Wiman-Valiron method}, Canadian Mathematical Bulletin \textbf{17}(3) (1974), 317--358.
\bibitem{He-Xiao} Y. He and X. Xiao, \textit{Algebroid functions and ordinary differential equations} (in Chinese), Science Press, Beijing (1988).\bibitem{HN} W. Helmrath and J. Nikolaus, \textit{Ein elementarer Beweis bei der Anwendung der Zentralindexmethode auf Differentialgleichungen}, Complex Variables Theory Appl. \textbf{3} (1984), 253--262.
\bibitem{IY} K. Ishizaki and N. Yanagihara, \textit{Wiman-Valiron method for difference equations}, Nagoya Mathematical Journal, \textbf{175} (2004), 75--102.
\bibitem{Jank-Volkmann} G. Jank and L. Volkmann, \textit{Einf\"{u}hrung in die Theorie der ganzen und meromorphen Funktionen mit Anwendungen auf Differentialgleichungen}, Birkh\"{a}user, Basel (1985).
\bibitem{Koekoek-Swarttouw} R. Koekoek, P. A. Lesky and R. F. Swarttouw, \textit{Hypergeometric orthogonal polynomials and their $q$-analogues}, Springer Monographs in Mathematics (2010).
\bibitem{Kovari1} T. K\"{o}vari, \textit{On the theorems of G. P\'{o}lya and P. Turan}, Journal d'Analyse Math\'ematique \textbf{6} (1958), 323--332.
\bibitem{Kovari2} T. K\"{o}vari, \textit{On the Borel exceptional values of lacunary integral functions}, Journal d'Analyse Math\'ematique \textbf{9} (1961), 71--109.
\bibitem{MN} S. Matsumoto and J. Novak, \textit{Jucys-Murphy elements and unitary matrix integrals}, International Mathematics Research Notices \textbf{2} (2013), 362--397.
\bibitem{Noerlund} N. E. N\o rlund, \textit{Le\c{c}ons sur les s\'{e}ries d'interpolation}, Gauthier-Villars, Paris (1926).
\bibitem{Saxer} W. Saxer, \textit{\"{U}ber die Picardschen Ausnahmewerte sukzessiver Derivierten}, Mathematische Zeitschrift \textbf{17} (1923), 206--227.
\bibitem{Valiron1} G. Valiron, \textit{Sur les fonctions enti\`{e}res d'ordre fini et d'ordre nul, et en particulier les fonctions \`{a} correspondance reguli\`{e}re}, Annales de la Facult\'e des Sciences de Toulouse (3) \textbf{5} (1913), 117--257.
\bibitem{Valiron2} G. Valiron, \textit{Sur le maximum du module des fonctions enti\`{e}res}, Comptes Rendus de l'Acad\'emie des Sciences (Paris) \textbf{166} (1918), 605--608.
\bibitem{Valiron} G. Valiron, \textit{Les th\'{e}or\`{e}mes g\'{e}n\'{e}raux de M. Borel dans la th\'{e}orie des fonctions enti\`{e}res}, Annales scientifiques de l'\'{E}cole Normale Sup\'{e}rieure (3) \textbf{37} (1920), 219--253.
\bibitem{Valiron3} G. Valiron, \textit{Lectures on the general theory of integral functions} (reprinted), Chelsea (1949).
\bibitem{Wiman1} A. Wiman, \textit{\"{U}ber den Zusammenhang zwischen dem Maximalbetrage einer analytischen Funktion und dem gr\"{o}ssten Gliede der zugeh\"{o}rigen Taylorschen Reihe}, Acta Mathematica \textbf{37} (1914), 305--326.
\bibitem{Wiman2} A. Wiman, \textit{\"{U}ber den Zusammenhang zwischen dem Maximalbetrage einer analytischen Funktion und dem gr\"{o}ssten Betrage bei gegebenem Argumente der Funktion}, Acta Mathematica \textbf{41} (1916), 1--28.
\end{thebibliography}

\end{document}